\documentclass[12pt]{article}

\usepackage[english]{babel}
\usepackage{amssymb}
\usepackage{graphics}
\usepackage{amsmath,stackengine}

\usepackage[numbers]{natbib}

\newcommand{\qed}{$\;\;\;\Box$}
\newenvironment{proof}{\par\smallbreak{\sl Proof.~}}
{\unskip\nobreak\hfill \qed \par\medbreak}
\newcounter{claim}[section]
\renewcommand{\theclaim}{\arabic{claim}}
\newenvironment{claim}{\refstepcounter{claim}%
	\par\medskip\par\noindent{\bf Claim~\theclaim.}\rm}%
{\par\medskip\par}

{\qed\par\smallbreak}

\newcommand{\hide}[1]{}

\newcommand{\E}{{\cal E}}

\newcommand{\N}{{\mathbb N}}

\newcommand{\R}{{\mathbb R}}

\newcommand{\Z}{{\mathbb Z}}

\newcommand{\beq}{\begin{equation}}
\newcommand{\ee}{\end{equation}}

\renewcommand{\d}{\partial}

\newtheorem{thm}{Theorem}[section]
\newtheorem{lemma}[thm]{Lemma}

\newtheorem{defn}[thm]{Definition}

\newtheorem{rem}{Remark}[section]

\newcommand{\al}{\alpha}

\newcommand{\ga}{\gamma}

\newcommand{\eps}{\varepsilon}
\newcommand{\vphi}{\varphi}

\newcommand{\om}{\omega}

\newcommand{\reff}[1]{(\ref{#1})}
\newcommand{\dd}{\!\;\mathrm{d}}

\date{}
\title{An Inverse Almost Periodic Problem for a Semilinear Strongly Damped Wave Equation}

\newcounter{thesame}
\author{
	Irina Kmit
	\thanks{Institute of Mathematics, Humboldt University of Berlin. On leave from the
		Institute for Applied Problems of Mechanics and Mathematics,
		National Academy of Sciences of	Ukraine, Lviv, Ukraine.
		{\small   E-mail:
			{\tt irina.kmit@hu-berlin.de}}}
	\ \ \ Nataliya Protsakh
	\thanks{
		Department of Computational Mathematics and Programming,  Institute of Applied Mathematics and Fundamental Sciences,
		Lviv Polytechnic National University,  Ukraine.
		{\small   E-mail:
			{\tt protsakh@ukr.net}}}
	\ \ \ Viktor Tkachenko
	\thanks{
		Institute of Mathematics,
		National Academy of Sciences of	Ukraine, Kyiv, Ukraine.
		{\small   E-mail:
			{\tt vitk@imath.kiev.ua}}}}

\begin{document}
	\maketitle
	
	\begin{abstract}
		\noindent
		This paper investigates an inverse boundary value problem for a semilinear strongly damped wave equation with Dirichlet boundary conditions in Sobolev spaces of functions bounded  in time on $\R$, including periodic and almost periodic functions. In addition to constructing a bounded strong solution, we determine a time-dependent source coefficient via an integral overdetermination condition ensuring well-posedness.
		After reducing the inverse problem to a direct one, we first establish existence and uniqueness of solutions to an associated problem on finite time intervals. We then extend these solutions to half-lines and construct a bounded strong solution on the whole real line as a limit of such extensions, and subsequently establish its uniqueness.
		In particular,  periodic and almost periodic data yield periodic and almost periodic solutions.

	\end{abstract}

	\emph{Key words:} semilinear strongly damped  wave equation, Dirichlet boundary conditions, inverse problem,
	integral 	overdetermination condition, bounded weak solutions, almost periodic and periodic solutions

\section{Introduction}
\renewcommand{\theequation}{{\thesection}.\arabic{equation}}

\subsection{Problem setting and main result}

\paragraph{Our setting} We consider boundary value problems for a strongly damped wave equation  of the  type
 \begin{equation}\label{1}
	u_{tt} -  a^2 \Delta u - b(t) \Delta u_t	
	+ \varphi(x,u)  = g(t) f_1(x)+ f_2(t,x),\qquad t\in\mathbb{R},\ x\in\Omega, 
\end{equation}
with Dirichlet boundary condition
\begin{equation}\label{2}
u|_{\partial \Omega \times \mathbb R} = 0,
\end{equation}
where $a\ne0$, $b(t)>0$  for all $t\in\R$, and $\Omega \subset \mathbb R^n$ ($n \in \mathbb N$) is a bounded domain with boundary of class $C^2$.  By $\Delta$ we denote the spatial Laplacian.

In spaces of bounded functions (in particular, periodic and almost periodic ones), we aim at identifying conditions that guarantee the existence and uniqueness of both the unknown function $u(t,x)$ and the coefficient $g(t)$. To this end, we supplement the system \reff{1}, \reff{2} with the so-called overdetermination condition
\begin{equation}\label{3}
\int_{\Omega } K(x)u(t,x)\dd x = E(t),\quad t\in\R.
\end{equation}

\paragraph{Notation} We  use  $\|\cdot\|$ to denote the maximum norm in $\R^n$, 
defined by $\|A\|=\max_{j\le n}|A_j|$ for $A\in\R^n$,
and 
 write $\alpha \cdot \beta$ for the scalar product of  $\alpha, \beta\in \R^n$.

Let  $X$ be a Banach space  with norm $\|\cdot\|_X$. For an interval $J \subseteq\mathbb R$,
we denote by $BC(J;X)$  the space of all bounded and continuous functions $u: J \to X$, and  write $BC(J)$ when $X=\R$. 
For $k\ge 1$,  $BC^k(J; X)$  consists of functions in $BC(J;X)$ with bounded derivatives up to order $k$. 

Given a domain $G\subset\R^n$, we  use standard notation for Lebesgue and Sobolev spaces, as well as for
 spaces of continuous functions, such as
$L^\infty(G),$
$L^2(G)$, $W^{p,k}(G)$,
$ H^k(G)=W^{k,2}(G)$,  and
$C^k(G)$.  By $C_c^k(G)$
we denote   the space of compactly supported $C^k(G)$-functions.

As usual, the gradient operator~$\nabla$ applied to an $n$-dimensional scalar function $u$ is the vector-function $\nabla u=(\partial_{x_1}u,\dots,\partial_{x_n}u).$
Write $Q = \mathbb R\times\Omega$, and $Q_{(t_1,t_2)} =  (t_1,t_2)\times\Omega $ for any $t_1,t_2 \in \mathbb R$ with $t_1<t_2.$

The symbols $\overset{}{\rightharpoonup}$ and $\overset{\ast}{\rightharpoonup}$ denote weak and weak-$*$ convergence, respectively.

\paragraph{Main result} Suppose that the coefficients of
 (\ref{1}) and (\ref{3}) satisfy the following conditions.
\begin{itemize}
	\item[{\bf (A1)}]
The function $b(t)$ belongs to $C^1(\R)$
and
satisfies the inequalities $0<\underline{b} \le b(t) \le \overline{b}$ and
$|b^\prime(t)| \le b_1$ for all $t \in \mathbb R$, where
 $ \underline{b}$, $\overline{b}$, and $b_1$ are positive reals.
\item[{\bf (A2)}]
 The function $\varphi(x,\xi)$ is measurable in  $x \in \Omega$  for all $\xi\in\mathbb{R}$ and continuously
 differentiable in $\xi\in \mathbb R$ for almost every $x\in\Omega$. Moreover, there exists a constant $L_0>0$ such
  that,
 for almost every $x \in \Omega$ and all $\xi, \eta \in \mathbb R,$  it holds
 $$\varphi(x,0) = 0 \quad \mbox{and}\quad
 |\varphi(x,\xi) - \varphi (x,\eta)| \le L_0 |\xi - \eta|.$$
\item[{\bf (A3)}]
$f_1\in L^2(\Omega)$ and $f_2\in BC^1(\mathbb R;L^2(\Omega))$.
\item[{\bf (A4)}]
   $E\in BC^2(\mathbb R)$ and $K\in  H^1_0(\Omega)$.
\item[{\bf (A5)}] $  \displaystyle\int_\Omega K(x)f_1(x)\dd x\ne 0$.
\end{itemize}

\begin{defn} \label{sol_weak}
	A pair of functions $(u(t,x),g(t))$  is  called a   {\rm  bounded strong solution} to the problem (\ref{1})--(\ref{3}) in $Q$ if
\beq\label{ug-reg}
\begin{array}{cc}
	u \in BC\left(\mathbb R; H_0^1(\Omega)\cap H^2(\Omega)\right), \quad
u_t \in BC\left(\mathbb R; H_0^1(\Omega)\right) \cap L^\infty\left(\mathbb R; H^2(\Omega)\right),
\\ [2mm]u_{tt}\in L^\infty\left(\R;L^2(\Omega)\right), 
\quad
g\in BC(\mathbb R ),
\end{array}
\ee
and $u$ fulfills the equation (\ref{1}) almost everywhere and the overdetermination condition
\reff{3} pointwise.

\end{defn}

We are prepared to formulate the main result of the paper.
\begin{thm} \label{main}
Let Conditions (A1)--(A5) be satisfied. Then the following statements hold.
\begin{itemize}
	\item[1.]
	If
	$L_0$, $b_1$, and $\|\nabla K\|_{L^2(\Omega)}$ are sufficiently small, then the problem (\ref{1})--(\ref{3}) has a unique  bounded  strong solution.
	\item[2.]
If, additionally, $K\in H^2(\Omega)$, $E\in BC^3(\mathbb R)$
(respectively, $E\in BC^2(\mathbb R)$), and the  coefficients of \reff{1} and \reff{3} are Bohr almost periodic (respectively, $\omega$-periodic) 
in~$t$, then the   bounded strong solution
 is Bohr almost periodic (respectively, $\omega$-periodic) in $t$ also.	
\end{itemize}

\end{thm}

The paper is organized as follows.
In Subsection \ref{remarks}, we  comment on the problem and discuss possible generalizations. Subsection \ref{motivation} is devoted to the motivation for this line of research and to a literature review, with particular emphasis on direct and inverse problems in the context of periodic and almost periodic solutions.

The proof of the main result (Theorem \ref{main}) is given in Section \ref{proof}. More precisely, Subsection \ref{equivalent} reformulates the inverse problem as an equivalent direct boundary value problem. Subsection \ref{proof_bounded} establishes bounded strong solvability, while Subsections \ref{ap} and \ref{per} prove the existence and uniqueness of strong almost periodic and periodic solutions, respectively.

Finally, Section \ref{extensions} presents further extensions of our results to inverse problems in bounded and unbounded domains of $Q$, illustrating the generality of the proposed approach.

\subsection{A couple of remarks}\label{remarks}
\paragraph{About the equation \reff{1}}
Our analysis extends to explicitly time-dependent nonlinearities $\vphi(t,x,u)$ and additional weak damping terms involving $u_t$. 
For simplicity, we omit these extensions in the present paper.

Another point concerns the role of $f_2$ in our setting.
This function represents external influences that are known in advance.
It may include noise,  environmental effects, or additional deterministic sources and, hence,  allow the model to accommodate more general data
(covering  forcings that are not exactly separable in $x$ and $t$). Nevertheless, mathematically it courses no loss of generality, as
it does not affect the identification of~$g(t)$.

\paragraph{Exponential stability and smallness assumptions}
The smallness assumptions on the data $L_0$, $b_1$, and $\|\nabla K\|_{L^2(\Omega)}$  in Theorem \ref{main} 
arise from the method used to prove the main result. They ensure exponential stability of solutions to the corresponding initial-boundary value problems, which in turn yields existence, uniqueness, and  (almost) periodicity of bounded strong solutions. We do not address here the necessity of the smallness conditions.

Although the smallness assumptions are not stated explicitly in Theorem \ref{main}, the proof provides constructive sufficient conditions.

	\paragraph{Inverse problems on finite time intervals and on  half-lines} 
In addition to the inverse problem on $\R$ in time, we prove well-posedness in a strong sense on finite time intervals and on half-lines.
 We therefore consider Theorem~\ref{T-obm-neobm} to be of independent interest. It is worth noting that the well-posedness of inverse problems on finite time intervals, in contrast to half-lines  and the whole axis, does not require smallness assumptions on the data.

\subsection{Motivation and state of the art}\label{motivation}

The aim of this paper is to establish existence and uniqueness of periodic and almost periodic solutions for an inverse problem associated with a strongly damped semilinear wave equation. In the model under consideration, both the state $u(t,x)$ and the time-dependent coefficient $g(t)$ in the right-hand side are unknown. Unlike most works on inverse problems for wave-type equations, which are formulated on finite time intervals with prescribed or unknown initial data, we work on an unbounded time domain and study bounded, periodic, and almost periodic solutions, which leads to new analytical challenges.

Strongly damped wave equations arise naturally in viscoelasticity, thermoelasticity, and structural acoustics, describing longitudinal vibrations of a homogeneous bar or transverse oscillations of a string under viscous effects. The damping term $b\Delta u_t$ accounts for stress depending not only on strain but also on the strain rate, as in linearized Kelvin–Voigt materials \cite{BMR,S}.

Inverse problems for such equations are significantly less developed than those for parabolic or weakly damped hyperbolic models. Earlier studies mainly addressed backward or initial value problems and their ill-posedness \cite{tuan17,tuan18}, while more recent works focus on coefficient or source identification on finite time intervals \cite{Li,Protsakh3}. However, the regime of almost periodic or time-periodic behavior  has not yet been considered in the inverse setting.

It is worth emphasizing that relatively few studies are devoted to time-periodic or almost periodic solutions, even in the context of direct problems (see, e.g., \cite{celik,Cuevas,Dvornyk,kokocki,Hongyan}). The study of such regimes is closely related to the long-time dynamics of strongly damped wave equations and the theory of global attractors (see, e.g., \cite{Joelma,Carvalho1,Chen,chueshov,KZ,Massatt2}).

Although strong damping suppresses classical resonance phenomena, the emergence of (almost) periodic solutions indicates a balance between (almost) periodic forcing and the dissipative and other stabilizing mechanisms inherent in the system.
In contrast, the study of periodic and almost periodic solutions for weakly damped wave equations is more subtle, due to the possible occurrence of resonances. To rule out such effects, additional assumptions are typically required. For results on periodic and almost-periodic solutions of weakly damped wave equations, we refer the reader, for example, to \cite{AP,BH,Feireisl,KmRe,Nakao2} and the references therein.

Monographs \cite{BY,lesnic} provide an extensive account of inverse problems with applications in science and engineering, encompassing models based on  damped wave equations.
In applications, materials and structures are often subjected to long-time periodic or quasi-periodic external forces. Identifying system parameters from such repeated excitations becomes essential for diagnostics, monitoring, and control.
Recovering the temporal factor~$g(t)$ is of particular interest, as it
represents the time modulation of a known spatially localized action 
$f_1(x)$ and determines the periodic or almost periodic response of the system. 

Integral-type overdetermination conditions, 
which are needed due to unknown $g$, are natural in applications 
where pointwise-in-time measurements are impractical, whereas cumulative quantities, such as total displacement, energy, stress, or time-averaged boundary data, are more accessible. Mathematically, these conditions 
provide improved stability and robustness. Recent works on inverse problems for strongly or 
weakly damped wave and pseudoparabolic equations have employed such 
conditions to recover unknown sources or coefficients \cite{Aitzhanov,Akhundov,Khonatbek,Mehraliyev,Oussaeif,Protsakh2,Protsakh3,Shahrouzi1,Shahrouzi2}.

\section{Proof of Theorem \ref{main}}\label{proof}

The proof of Theorem \ref{main} is divided into several steps. In the first step, we propose an equivalent formulation 
of the inverse problem \reff{1}--\reff{3} as a direct problem (Subsection~\ref{equivalent}), 
introducing a technical parameter $\al>0$. 
The proof of Claim 1 of Theorem \ref{main}, concerning bounded solvability, is presented in Subsection \ref{proof_bounded} and consists of three parts. First, we prove that an associated initial-boundary value problem admits a unique strong solution on any finite time interval. We then extend these solutions in a bounded way to the positive time half-lines. Finally, we construct a bounded strong solution on the entire real line $\R$ as the limit of solutions defined on the positive  half-lines  and show the uniqueness of the constructed solution.
We emphasize that (and this is a crucial point of the construction),  although the solutions 
 on finite time intervals and on half-lines  depend on the introduced parameter  $\al$, the limit function,
 which solves the original problem on $\R$, is independent of~$\al$.
The proof of Claim 2 of Theorem~\ref{main}, establishing almost periodicity and periodicity, is carried out in 
Subsections~\ref{ap} and \ref{per}.

\subsection{Equivalent formulation  of the problem \reff{1}--\reff{3}}\label{equivalent}
\renewcommand{\theequation}{{\thesection}.\arabic{equation}}
\setcounter{equation}{0}

In this section, we provide an equivalent  formulation of the problem (\ref{1})--(\ref{3}),
  which will be essential in the proof of the main result.

We introduce the following notation, consistent with Condition (A5):
$$
K_0(x) = f_1(x)\left(\int_{\Omega}K(\xi)f_1(\xi)\dd\xi\right)^{-1}.
$$
\begin{lemma}\label{lemma2.2}
    A pair of functions $(u(x,t), g(t))$  is  a  bounded strong solution to the problem  (\ref{1})--(\ref{3}) if and only if
    it satisfies the regularity conditions \reff{ug-reg},
\beq\label{44}
 \begin{array}{rcl}
	    g(t) &=&\displaystyle  \left(\int_{\Omega}K(x)f_1(x)\dd x\right)^{-1}
	\bigg[E''(t) + \int_\Omega \Bigl( a^2\nabla K(x) \cdot  \nabla u \\ [3mm]
	& &                   + b(t)\nabla K(x) \cdot \nabla u_{t}
	+ K(x)\varphi(x,u) - K(x)f_2(t,x)\Bigl)\dd x\bigg]\quad\mbox{for all }\,  t\in \mathbb{R},
\end{array}
\ee
and the  equation
\begin{equation} \label{4}
\displaystyle u_{tt} - a^2\Delta u -  b(t)\Delta  u_t	
	+  \varphi(x,u)  +\al K_0(x)\left(E(t)-\int_{\Omega } K(\xi)u(t,\xi)\dd\xi  \right)=g(t)  f_1(x)+ f_2(t,x) ,  
\end{equation}
holds almost everywhere  in $Q$
for some $\al>0$. 
    \end{lemma}

\begin{proof}{\it Necessity. }
Let $(u, g)$ be a bounded strong solution to the problem
    (\ref{1})--(\ref{3}). After differentiating  (\ref{3}) twice, substituting  $u_{tt}$ from
   (\ref{1}), and performing integration by parts, we obtain the identity
$$
\begin{array}{ll}
\displaystyle g(t)\int_{\Omega}K(x)f_1(x)\dd x\\ [4mm]
\displaystyle \quad+\int_{\Omega}\left[K(x)\left(f_2(t,x)
-\varphi(x,u) \right)     - a^2\nabla K(x) \cdot  \nabla u -
b(t)\nabla K(x) \cdot \nabla u_{t}
\right]\dd x=   \frac{\dd^2}{\dd t^2}E(t).     
\end{array}
  $$      
        Combining it with  Condition
        (A5) yields the formula~\reff{44} for~$g(t)$.
        
        The equation \reff{4} (for any real $\al$) follows directly from the equation \reff{1} and the integral condition~\reff{3}.

 {\it Sufficiency.} Let $(u^*,g^*)$ be a pair of functions satisfying the regularity conditions \reff{ug-reg}
for $u=u^*$, $g=g^*$,   and for some constant $\al>0$.
  Set
  $$
  E^*(t)=\int_{\Omega}K(x)u^*(t,x)\dd x
  $$
  and show that $E^*(t)\equiv E(t)$ for all $t\in\R$.  To this end, we differentiate 
$E^*(t)$ twice and, using \reff{4}, derive the equality
   \begin{eqnarray*}
  & &    \frac{\dd^2}{\dd t^2}{E^*}(t) = \int_{\Omega}K(x)u_{tt}^*\dd x=  \int_\Omega \Big[g^*(t)f_1(x)+f_2(t,x)\Big]K(x)\dd x\\
 && \quad  -\int_\Omega \Bigl[ a^2\nabla K(x) \cdot  \nabla u^*
    +b(t)\nabla K(x) \cdot \nabla u_{t}^*
            +K(x)\varphi(x,u^*)\Bigl]\dd x \\
         & & \quad   -\al \left(E(t)-\int_{\Omega } K(\xi)u^*(t,\xi)\dd\xi \right)   
                 =   
              \frac{\dd^2}{\dd t^2}E(t) \\ [3mm]
              & &\quad    + \int_\Omega \Bigl[a^2\nabla K(x) \cdot  \nabla u^*                + b(t)\nabla K(x) \cdot \nabla u_{t}^*
             + K(x)\varphi(x,u^*) - K(x)f_2(t,x)\Bigl]\dd x\\
             &&\quad+\int_\Omega  f_2(t,x)K(x)\dd x
                -\int_\Omega \Bigl[ a^2\nabla K(x) \cdot  \nabla u^*
        +b(t)\nabla K(x) \cdot \nabla u_{t}^*
        +K(x)\varphi(x,u^*)\Bigl]\dd x\\
        &&\quad-\al \left(E(t)-\int_{\Omega } K(\xi)u^*(t,\xi)\dd\xi \right)=  \frac{\dd^2}{\dd t^2} E(t) +\al
        \left(E^*(t) - E(t)\right),
   \end{eqnarray*}
which imply  the identity
     \beq\label{E-E}   
     \frac{\dd^2}{\dd t^2} (E^*(t) - E(t)) = \al(E^*(t) - E(t)) \qquad \mbox{for all }   t\in \mathbb R.
 \ee
   As  $E^*(t)$ and $E(t)$ are bounded on $\R$ and $\al> 0$, it follows that
	$E^*(t)=E(t)$ for all $  t \in \mathbb R.$  Therefore,
	$u^*$ satisfies the overdetermination condition (\ref{3}) and, hence the formulation  of the original problem,  as desired. 
\qed	\end{proof}

\begin{rem}\label{all-al}\rm
The parameter~$\al>0$ is introduced only for technical purposes. As we have seen, it plays a crucial role in the equivalent reformulation of the original problem provided in Lemma~\ref{lemma2.2}.
Moreover, as follows from the sufficiency part of the proof of Lemma~\ref{lemma2.2}, although equation~\reff{4} depends on the parameter $\alpha$, the bounded strong solution to the problem \reff{44}, \reff{4}, \reff{2}, and hence to the original problem \reff{1}--\reff{3}, is independent of~$\alpha$. In the proof of the bounded strong solvability in Section \ref{proof_bounded},  the parameter 
$\al$ will be assumed sufficiently small.
\end{rem}

\subsection{Bounded  strong solutions}\label{proof_bounded}

\subsubsection{Auxiliary problem  on  finite time intervals}\label{interval}
By combining (\ref{44}) and (\ref{4}), we obtain the equation
\begin{equation}
	\displaystyle u_{tt} -  a^2\Delta u  - b(t) \Delta u_t 	
	+  \varphi(x,u) = \left[\Phi_\al u\right](t,x) + F_\al(t,x),   \label{G0}
\end{equation}
where
\beq \label{FPhi}  
\begin{array} {rcl}
	\displaystyle  
	F_{\al}(t,x)&=&\displaystyle f_2(t,x)+K_0(x)\left(E''(t)-\al E(t)-\int_\Omega K(\xi)f_2(t,\xi)\dd \xi\right),\\
	\left[\Phi_\al u\right](t,x)&=&\displaystyle K_0(x)\int_\Omega \Bigl[a^2\nabla K(\xi) \cdot  \nabla u+ b(t)\nabla K(\xi) \cdot \nabla u_{t}\\ [3mm]
	&&\qquad\qquad\ \ + K(\xi)\varphi(\xi,u)+\al K(\xi)u\Bigr]\dd \xi.
\end{array}
\ee

In this section, we investigate, on any finite time interval, the existence and uniqueness of a   strong solution to the
problem \reff{G0},  \reff{2}, with the initial conditions
\begin{equation}\label{2p}
	u(s,x)=u_0(x),\quad u_t(s,x)=u_1(x) \qquad \mbox{for all } x\in\Omega
\end{equation}
at time $t=s$, for an arbitrary fixed initial time $s\in\R$.

\begin{defn}
\label{def-obm}
Let  $\al\ge 0$ and $s, T\in\R$ be arbitrary real numbers  satisfying  $s< T$.  A  function $u$  is  called a  {\rm strong solution}
to the initial-boundary value problem  (\ref{G0}), \reff{2},  \reff{2p} in $Q_{(s,T)}$
 if 
$$
\begin{array}{cc}
u \in C\left([s,T]; H_0^1(\Omega)\cap H^2(\Omega)\right),\quad 
u_t \in C\left([s,T]; H_0^1(\Omega)\right)\cap L^\infty\left(s,T;H^2(\Omega)\right), \\ 
[2mm] u_{tt}\in L^\infty\left(s,T;L^2(\Omega)\right),
\end{array}
$$
and $u$ satisfies both the equation (\ref{G0}) and the initial conditions  (\ref{2p})  almost everywhere.
\end{defn}

In order to prove the main result of this subsection, we also introduce the notion of a  weak solution.

\begin{defn}
\label{def-obm2}Let $\al\ge 0$ and $s, T\in\R$ be arbitrary real numbers  satisfying  $s< T$. 	 A  function $u$    is  called a  {\rm  weak solution}
to the initial-boundary value problem  (\ref{G0}), \reff{2},  \reff{2p} in $Q_{(s,T)}$
 if 
$$u \in C\left([s,T]; H_0^1(\Omega)\right), \quad u_t \in C\left([s,T]; H_0^1(\Omega)\right), \quad 
u_{tt}\in L^\infty\left(s,T;L^2(\Omega)\right),
$$
and the following equality holds for  all   $v\in L^2\left(s,T ;H_0^1(\Omega)\right)$:
\begin{equation}
 \displaystyle\int\limits_{Q_{(s,T)}} \Bigl(u_{tt}v +  a^2\nabla u \cdot \nabla v + b(t) \nabla u_t \cdot \nabla v	
 +  \varphi(x,u)v -\left[\Phi_\al u\right](t,x)v- F_\al(t,x)v\Bigl)\dd x\dd t
   = 0.\label{G1}\end{equation}
   \end{defn}

In what follows, $\gamma>0$ will denote the constant in Friedrich's inequality
\begin{equation}
	\|u\|_{L^2(\Omega)} \le \gamma  \|\nabla u\|_{L^2(\Omega)}, 
	\label{riv16}
\end{equation}
which holds uniformly for all $u\in H_0^{1}(\Omega)$.

\begin{thm}
\label{T-obm}
Let $s, T\in\R$ be arbitrary real numbers  satisfying  $s< T$. Moreover,  let 	$u_0\in H_0^1(\Omega)\cap H^2(\Omega)$,
$u_1 \in H_0^1(\Omega)\cap H^2(\Omega)$, and   Conditions {(A1)}--{(A5)} be satisfied. Then there exists $\al_1>0$ such that, for all  $\al\in[0,\alpha_1] $,
 the problem  (\ref{G0}), \reff{2},  \reff{2p} in $Q_{(s,T)}$
admits a unique  strong solution.
\end{thm}
\begin{proof}
	The proof is  divided into a number of steps. First, we prove the existence and uniqueness of a weak solution to
	 the problem. Then, after establishing higher regularity, we obtain a strong solution.

	{\it 1. Galerkin approximations and energy estimates.} 
Let $(w^k)$ be a basis of $H_0^1(\Omega) \cap H^2(\Omega)$ that is
orthonormal in $L^2(\Omega).$ Then, for every $N\ge 1$, there exist sequences $(u^N_{0,k})_{k=1}^\infty$ and $(u^N_{1,k})_{k=1}^\infty$
such that the sequences $(u_0^N)_{N=1}^\infty$ and $(u_1^N)_{N=1}^\infty$ defined by
$$
u^N_0(x) = \sum\limits_{k=1}^N u^N_{0,k}w^k(x) \quad \mbox{and}\quad u^N_1(x) = \sum\limits_{k=1}^N u^N_{1,k}w^k(x),
$$
 approximate the initial functions $u_0$ and $u_1$ in the sense of the following convergences:
\begin{equation}\label{zb}
	\lim_{N \to \infty}\|u_0 - u_0^N\|_{H^2(\Omega)} = 0, \quad
	\lim_{N \to \infty}\|u_1 - u_1^N\|_{H^2(\Omega)} = 0.
\end{equation}

Write
\beq\label{uN}u^N(t,x) = \sum_{k=1}^N c_k^N(t) w^k(x),
\ee
where, for every $N\ge 1$, the vector of coefficients $(c_1^N(t),\dots, c_N^N(t))$ is a solution to the following  linear $N\times N$-system of ordinary differential equations:
\begin{equation} \label{G2}
\begin{array}{cc}
 \displaystyle	 \int_\Omega \Bigl(u_{tt}^N w^k + a^2\nabla u^N \cdot\nabla w^k +
b(t) \nabla u^N_t \cdot\nabla w^k	+ \varphi(x,u^N) w^k \Bigl) \dd x \\
	= \displaystyle  \int_\Omega \left[\Phi_\al u^N\right](t,x)w^k\dd x+  \int_\Omega  F_\al(t,x) w^k \dd x, \qquad k\le N,
\end{array}
\end{equation}
subjected to the initial conditions
\begin{equation} \label{G2-in}
	\begin{array}{cc}
		\displaystyle \qquad c_k^N(s) = u^N_{0,k}, \quad c_{kt}^N(s) = u^N_{1,k}, \qquad k \le N.
	\end{array}
\end{equation}
By applying Carath\'{e}odory's Theorem \cite[p. 43]{Coddington} to (\ref{G2}), we state that there exists $T_0\le T$
such that the problem (\ref{G2})--(\ref{G2-in}) has a unique absolutely continuous solution $(c_1^N(t),\dots, c_N^N(t))$ on the interval $[s,T_0]$. Hence,
the function $u^N$, given by \reff{uN}, satisfies both the equations (\ref{G2}) for all $k\le N$ and $t\in [s,T_0]$, together with the initial conditions
$$
u^N(s,x) = u^N_{0}(x), \quad u_t^N(s,x) = u^N_{1}(x), \qquad k \le N.
$$

We now show that $u^N$ extends to $[s,T]$ and satisfies  (\ref{G2}) on this interval. To this end, we combine the local existence result on $[s,T_0]$  with the global a priori estimates for~$u^N$, which will be established below. To derive the a priori estimates, we multiply the $k$-th equation of (\ref{G2}) by $(c_{k}^N(t))'$, 
sum  over $k=1,\dots, N$, and integrate  over  $[s,\tau]$, where $ \tau\le T.$  This yields 
\beq\label{G3}
\begin{array}{cc}
 \displaystyle  \int_{Q_{(s,\tau)}}\Bigl(u_{tt}^N u_t^N + a^2\nabla u^N \cdot\nabla u_t^N +
b(t) \nabla u^N_t \cdot\nabla u_t^N
 \Bigl)\dd x\dd t\\	[4mm]
\displaystyle	=  \int_{Q_{(s,\tau)}}\Bigl(F_\al(t,x)- \varphi(x,u^N)+\left[\Phi_\al u^N\right](t,x)\Bigr)u^N_t\dd x \dd t,
\end{array}
\ee
and an integration by parts leads to
\beq\label{G4pr}
\begin{array}{cc} 
 \displaystyle 	 \frac{1}{2}\int_\Omega |u_{t}^N(\tau,x)|^2\dd x+
\frac{a^2}{2} \int_\Omega \|\nabla u^N(\tau,x)\|^2 \dd x
 + \int_{Q_{(s,\tau)}}  b(t)\|\nabla u^N_{t}\|^2 \dd x \dd t \\ [3mm]
 \displaystyle 
 	 = \frac{1}{2}\int_\Omega |u_{1}^N(x)|^2\dd x  +
  \frac{a^2}{2}\int_\Omega \|\nabla u^N_{0}(x) \|^2\dd x \\ [3mm]
 \displaystyle  + \int_{Q_{(s,\tau)}}\Bigl(F_\al(t,x)- \varphi(x,u^N)+\left[\Phi_\al u^N\right](t,x)\Bigr)u^N_t\dd x \dd t. 
\end{array}
\ee
The following estimate follows from  (\ref{FPhi}) and (\ref{riv16}):
\beq\label{phi} 
\begin{array} {rcl}
  \displaystyle\int_{Q_{(s,\tau)}}|\varphi(x,u^N)u^N_{t}|\dd x \dd t 
& \le&  \displaystyle\frac {L_0}{2}\int_s^{\tau}\left( \gamma^2\|\nabla u^N\|^2_{L^2(\Omega)} +
 \|u^N_{t}\|^2_{L^2(\Omega)}\right)\dd t, \\ [5mm]
 \left\|\left[\Phi_\al u^N\right](t,\cdot)\right\|_{L^2(\Omega)} &\le& \displaystyle
 \eps_1(\al) \left(\|\nabla u^N(t,\cdot)\|_{L^2(\Omega)}+\|\nabla u_t^N(t,\cdot)\|_{L^2(\Omega)}
\right), 
\end{array}
\ee
where
$$
\eps_1(\al)=\|K_0\|_{L^2(\Omega)}\max\Bigl\{a^2\|\nabla K\|_{L^2(\Omega)} +
\ga(\al+L_0 ) \|K\|_{L^2(\Omega)},\overline b\|\nabla K\|_{L^2(\Omega)}\Bigr\}.
$$
Combining \reff{G4pr} with \reff{phi} and with
 $$\displaystyle	 \int_{Q_{(s,\tau)}}\left|\left[\Phi_\al u^N\right](t,x)u^N_{t}\right|\dd x\dd t
 \le \frac{1}{2\eps_0}\int_s^\tau \left\|\left[\Phi_\al u^N\right](t,\cdot)\right\|^2_{L^2(\Omega)}\dd t+
\frac{\eps_0}{2}\int_s^\tau \left\|u^N_{t}(t,\cdot)\right\|^2_{L^2(\Omega)}\dd t, $$ 
being true for any $\eps_0>0$, we obtain
\beq
\begin{array}{ll} \label{G4}
  \displaystyle 	 \|u_{t}^N(\tau,\cdot)\|^2_{L^2(\Omega)} +
a^2 \|\nabla u^N(\tau,\cdot)\|^2_{L^2(\Omega)} +
 2\biggl(\underline b -\frac {\eps_1(\al)^2}{\eps_0}\biggr)\int_{s}^\tau  \|\nabla u^N_t(t,\cdot)\|^2_{L^2(\Omega)} \dd t \\ [5mm]
  \displaystyle  \displaystyle \quad\quad\	 \le  \|u_{1}^N \|^2_{L^2(\Omega)} +
a^2  \|\nabla u^N_0\|^2_{L^2(\Omega)} + 
 \int_s^\tau \| F_\al(t,\cdot)\|_{L^2(\Omega)}^2 \dd t\\ [5mm]
\  \displaystyle  
\quad\qquad\,+
\biggl(\frac {2\eps_1(\al)^2}{\eps_0}+L_0\ga^2\biggr)\int_s^\tau  \|\nabla u^N(t,\cdot)\|^2_{L^2(\Omega)}\dd t
 +(1+\eps_0+L_0) \int_s^\tau\|u^N_t(t,\cdot)\|^2_{L^2(\Omega)}\dd t.
\end{array}
\ee
Fix $\eps_0>0$ 
to ensure the inequality 
$ 2\eps_1(0)^2 /\eps_0<\underline b.
$
 Then there exists  a sufficiently small~$\al_1>0$, such that for all  $\al\in[0,\al_1]$ it holds
\beq\label{eps0}
\frac {2\eps_1(\al)^2}{\eps_0}<\underline b.
\ee
Moreover, in view of the convergences  (\ref{zb}),  there exists  a constant~$A_0$, independent of $N\in\N$, $s\in\R$,  $\tau\in[s,T]$, and $\al\in[0,\al_1]$, such that
$$  \| u^N_1\|^2_{L^2(\Omega)} +
a^2 \|\nabla u^N_0\|^2_{L^2(\Omega)}+\int_s^\tau \| F_\al(t,\cdot)\|_{L^2(\Omega)}^2 \dd t\le A_0.
$$
The inequality (\ref{G4}) then yields
\beq
\begin{array}{ll}\label{G44}
 \displaystyle
\|u_{t}^N(\tau,\cdot)\|^2_{L^2(\Omega)} +
 \|\nabla u^N(\tau,\cdot)\|^2_{L^2(\Omega)}+\int_{s}^\tau  \|\nabla u^N_t(t,\cdot)\|^2_{L^2(\Omega)} \dd t
 \\ [2mm]
 \displaystyle \quad	 \le A_1+
  A_2\int_s^{\tau}\left( \|u_{t}^N(t,\cdot)\|^2_{L^2(\Omega)} +
 \|\nabla u^N(t,\cdot)\|^2_{L^2(\Omega)} +\int_s^{t}\|\nabla u^N_t(\theta,\cdot)\|^2_{L^2(\Omega)} \dd \theta\right) \dd  t,
\end{array}
\ee
where
$$
\begin{array}{rcl}
	  A_1&=& A_0\left( \min\left\{1,a^2,\underline b \right\}\right)^{-1}
\end{array}
$$  
and the constant $A_2>0$ is fixed to satisfy the following inequality for all  $\al\in[0,\al_1]$:
$$
\begin{array}{rcl}A_2&\ge&\displaystyle\max\biggl\{\frac {2\eps_1(\al)^2}{\eps_0}+L_0\ga^2,1+\eps_0+L_0\biggr\}\left( \min\left\{1,a^2,\underline b \right\}\right)^{-1}.
\end{array}
$$
Note that both constants $A_1$ and  $A_2$ do not depend on $N$, $s$, and $T$.
By Gr\"onwall's inequality applied to (\ref{G44}),  we obtain the bound
$$
\|u_{t}^N(\tau,\cdot)\|^2_{L^2(\Omega)} +
\|\nabla u^N(\tau,\cdot)\|^2_{L^2(\Omega)} +\int_{s}^\tau  \|\nabla u^N_t(t,\cdot)\|^2_{L^2(\Omega)} \dd t \le A_1 e^{A_2 (T-s)}.
$$
Combined with Friedrich's inequality, this yields another bound, namely,
\beq\label{prod}
\begin{array}{rr}
\|u^N(\tau,\cdot)\|^2_{L^2(\Omega)} +\|u_{t}^N(\tau,\cdot)\|^2_{L^2(\Omega)}+\|\nabla u^N(\tau,\cdot)\|^2_{L^2(\Omega)} \\ [3mm]
 \displaystyle 
+\int_{s}^\tau  \|\nabla u^N_t(t,\cdot)\|^2_{L^2(\Omega)} \dd t \le A_1(1+\ga^2)e^{A_2 (T-s)}.
\end{array}
\ee

Next, we derive an upper bound
for $u^N_{tt}.$ To this end, we differentiate the $k$-th  equation in (\ref{G2}) with respect to  $t$,   multiply  by $(c_{k}^N(t))^{\prime\prime}$,   sum  over $k=1,\dots,N$, and integrate   
 over~$[s,\tau]$. Using \reff{uN}, we obtain
\beq
\begin{array}{cc} \label{G3tt}
 \displaystyle \int_{Q_{(s,\tau)}}\Bigl(u_{ttt}^N u_{tt}^N + a^2\nabla u^N_t \cdot \nabla u^N_{tt} +
b^\prime(t)\nabla u^N_t \cdot \nabla u^N_{tt} +
b(t)\nabla u^N_{tt} \cdot \nabla u^N_{tt}  \\ + \partial_2 \varphi(x,u^N) u^N_t   u^N_{tt}
 \Bigl) \dd x \dd t
 \displaystyle
 = \int_{Q_{(s,\tau)}}  \Bigl(\d_tF_\al(t,x)+\d_t\left[\Phi_\al u^N\right](t,x)\Bigr)u^N_{tt}  \dd x \dd t,
\end{array}
\ee
where
$$
\begin{array}{ll}
\displaystyle\d_t\left[\Phi_\al u^N\right](t,x) = \al K_0(x)\int_{\Omega}
K(\xi)u_t^N\dd\xi\\ \qquad +
\displaystyle K_0(x)  \int_{\Omega} \Bigl( (a^2 + b^\prime(t))\nabla K(\xi) \cdot \nabla u^N_{t}+ b(t)\nabla K(\xi) \cdot \nabla u^N_{tt} + K(\xi) \partial_2 \varphi(x,u^N) u^N_t \Bigl) \dd \xi.
\end{array}
$$
Here and in what follows, $\partial_j$ denotes the partial derivative with respect to the $j$-th argument.
The following estimate is straightforward:
\beq\label{der}
\left\| \d_t\left[\Phi_\al u^N\right](t,\cdot)\right\|_{L^2(\Omega)} \le 
\eps_2(\al)\left(\|\nabla u_t^N(t,\cdot)\|_{L^2(\Omega)} +
\| \nabla u^N_{tt}(t,\cdot)\|_{L^2(\Omega)}\right),
\ee
where 
$$
\eps_2(\al)=\|K_0\|_{L^2(\Omega)}\max\Bigl\{\overline b \|\nabla K\|_{L^2(\Omega)},
(a^2+b_1)\|\nabla K\|_{L^2(\Omega)}+L_0\ga\|K\|_{L^2(\Omega)}
+\al\ga\|K\|_{L^2(\Omega)}\Bigr\}.
$$
Moreover, on the account of \reff{der},  the following estimate:
\beq\label{G33}
\begin{array}{lll}
	\displaystyle	 \int_{Q_{(s,\tau)}}\left|\d_t\left[\Phi_\al u^N\right](t,x)u^N_{tt}\right|\dd x\dd t
\\[3mm]
\displaystyle\qquad	\le \frac{1}{2\eps_3}\int_s^\tau \left\|u^N_{tt}(t,\cdot)\right\|^2_{L^2(\Omega)}\dd t +\frac{\eps_3}{2}\int_s^\tau \left\|\d_t\left[\Phi_\al u^N\right](t,\cdot)\right\|^2_{L^2(\Omega)}\dd t \\[3mm]
\displaystyle\qquad\le \frac {1}{2\eps_3}\int_s^{\tau} \|u^N_{tt}(t,\cdot)\|^2_{L^2(\Omega)}\dd t 
	\displaystyle +\eps_2(\al)^2\eps_3\int_s^\tau 
	\left(\|\nabla u^N_t(t,\cdot)\|^2_{L^2(\Omega)}+\|\nabla u_{tt}^N(t,\cdot)\|^2_{L^2(\Omega)}
	\right)\dd t
\end{array}
\ee
holds for every $\eps_3>0$.

After integration by parts in \reff{G3tt}, we obtain
\beq\begin{array}{cc} \label{G3ttt}
\displaystyle \int_{\Omega}|u_{tt}^N(\tau,x)|^2\dd x
 + a^2 \int_{\Omega}\|\nabla u^N_t(\tau,x)\|^2 \dd x + 2 \int_{Q_{(s,\tau)}}\bigl(b(t)\|\nabla u^N_{tt}\|^2
 \\ [5mm]
 \displaystyle + b^\prime(t)\nabla u^N_t \cdot \nabla u^N_{tt} +\partial_2 \varphi(x,u^N) u^N_t  u^N_{tt}
 \bigl) \dd x \dd t  = \int_{\Omega}|u_{tt}^N(s,x)|^2\dd x\\ [3mm]
 \displaystyle 
 + a^2 \int_{\Omega}\|\nabla u^N_1(x)\|^2 \dd x + 2 \int_{Q_{(s,\tau)}}  \Bigl(\d_tF_\al(t,x)+\d_t\left[\Phi_\al u^N\right](t,x)\Bigr)u^N_{tt}  \dd x \dd t.
\end{array}
\ee
Combining \reff{G3ttt} with  \reff{G33} and with
$$
2 \int_{Q_{(s,\tau)}} b^\prime(t)\nabla u^N_t \cdot \nabla u^N_{tt}  \dd x \dd t\le
\eps_4\int_s^\tau\| \nabla u_{tt}^N(t,\cdot)\|^2_{L^2(\Omega)} \dd t+\frac{b_1^2}{\eps_4}  \int_s^{\tau}  \|\nabla u^N_{t}(t,\cdot)\|^2_{L^2(\Omega)}\dd t
$$
gives
\beq\label{G3ttt2}
\begin{array}{ll}
\displaystyle \| u_{tt}^N(\tau,\cdot)\|^2_{L^2(\Omega)}
+ a^2 \|\nabla u^N_t(\tau,\cdot)\|^2_{L^2(\Omega)} + \left(2\underline b-\eps_4-2\eps_2(\al)^2\eps_3\right) \int_s^\tau\| \nabla u_{tt}^N(t,\cdot)\|^2_{L^2(\Omega)} \dd t \\ [4mm]
\qquad \qquad \le\displaystyle
 \| u_{tt}^N(s,\cdot)\|^2_{L^2(\Omega)}
+ a^2 \|\nabla u^N_1 \|^2_{L^2(\Omega)}+\int_s^{\tau}   \| \partial_t F_\al(t,\cdot)\|_{L^2(\Omega)}^2\dd t\\ [3mm]
\ \displaystyle\qquad \quad \qquad +\biggl(1+L_0+\frac{1}{\eps_3}\biggr)   \int_s^{\tau}    \|u^N_{tt}(t,\cdot)\|^2_{L^2(\Omega)} \dd t \\ [4mm]
\qquad \qquad \quad\ \displaystyle+ \bigg(2 \eps_2(\al)^2\eps_3+L_0\gamma^2+\frac{b_1^2}{\eps_4}\bigg)  \int_s^{\tau}  \|\nabla u^N_{t}(t,\cdot)\|^2_{L^2(\Omega)}\dd t,
\end{array}
\ee
where $\eps_4$ is a positive real number. Suppose that $\eps_3$ and $\eps_4$ are chosen so that 
$\underline b>\eps_4+ 2\eps_2(0)^2\eps_3.$
Then the value of $\al_1$ can be chosen smaller, if necessary, so that for all $\al\in[0,\al_1]$, in addition to \reff{eps0}, the following inequality holds:
\beq\label{eps34}
\underline b>\eps_4+2\eps_2(\al)^2\eps_3.
\ee

To obtain the desired estimate for $u^N_{tt}$ from \reff{G3ttt}, it remains to estimate
 $\| u_{tt}^N(s,\cdot)\|^2_{L^2(\Omega)}$. To this end, we 
 multiply the $k$-th differential equation in \reff{G2} by $(c_{k}^N(t))^{\prime\prime}$,  sum  with respect to $k=1,\dots,N$, and 
evaluate the resulting equality at $t=s$. Consequently, we obtain
\begin{equation} \label{G20}
\begin{array}{lll}
 \displaystyle	 \int_\Omega \Bigl(|u_{tt}^N(s,x)|^2
 + a^2\nabla u^N_0(x)\cdot \nabla u_{tt}^N(s,x) +
b(s)\nabla u^N_1(x)\cdot \nabla u_{tt}^N(s,x)	
	 \\  \qquad\displaystyle	 + \varphi(x,u^N_0(x)) u_{tt}^N(s,x)
 \Bigl)\dd x 	= \int_\Omega \Bigl(F_\al(s,x)+\left[\Phi_\al u^N\right](s,x)\Bigr)u_{tt}^N(s,x) \dd x.
\end{array}
\end{equation}
Notice the following simple inequalities:
$$
\begin{array}{ll}
 \displaystyle\ \	 \left|\int_\Omega a^2\nabla u^N_0(x)\cdot \nabla u_{tt}^N(s,x)\dd x\right|  
 \ \le \ \displaystyle	\frac{a^4}{ 4\eps_5}\|\Delta u^N_0\|_{L^2(\Omega)}^2
+\eps_5\| u^N_{tt}(s,\cdot)\|_{L^2(\Omega)}^2,	\\ [4mm]
 \displaystyle \left|\int_\Omega b(s)\nabla u^N_1(x) \cdot\nabla u_{tt}^N(s,x)\dd x \right|
 \ \le\ \displaystyle\frac{\overline b^2}{4\eps_5}\|\Delta u^N_1\|_{L^2(\Omega)}^2
+\eps_5\| u^N_{tt}(s,\cdot)\|_{L^2(\Omega)}^2,	\\ [4mm]
\displaystyle
\quad\ \ \left|\int_\Omega
 \varphi(x,u^N_0(x)) u_{tt}^N(s,x)
\dd x\right|\ \le \ \frac{L_0^2}{4\eps_5}\|u_0^N\|_{L^2(\Omega)}^2+\eps_5\| u^N_{tt}(s,\cdot)\|_{L^2(\Omega)}^2,\\ [4mm]
\displaystyle
\left|
\int_\Omega \left(F_\al(s,x)+\left[\Phi_\al u^N\right](s,x)\right)u_{tt}^N(s,x) \dd x
\right| \\[4mm]
\displaystyle \qquad\qquad \le\ \frac{1}{4\eps_5}\int_\Omega \left(F_\al(s,x)+\left[\Phi_\al u^N\right](s,x)\right)^2\dd x
+\eps_5\| u^N_{tt}(s,\cdot)\|_{L^2(\Omega)}^2,
\end{array}
$$
being true for any $\eps_5>0$. Combining these with \reff{G20} yields
\begin{eqnarray*}
(1-4\eps_5) \|u_{tt}^N(s,\cdot)\|_{L^2(\Omega)}^2& \le&  \frac{1}{4\eps_5}\Bigl( a^4\|\Delta u^N_0\|_{L^2(\Omega)}^2  +  \overline b^2\|\Delta u^N_1\|_{L^2(\Omega)}^2+
L_0^2 \|u^N_0\|_{L^2(\Omega)}^2 \\
 & &  +
\|F_\al(s,\cdot)+\left[\Phi_\al u^N\right](s,\cdot)\|_{L^2(\Omega)}^2\Bigl), 
\end{eqnarray*}
where
\begin{eqnarray*}
& & \left[\Phi_\al u^N\right](s,x)=\displaystyle K_0(x)\int_\Omega \Bigl[a^2\nabla K(\xi) \cdot  \nabla u_0^N \\
& & \qquad\qquad+ b(t)\nabla K(\xi) \cdot \nabla u_{1}^N
+ K(\xi)\varphi(\xi,u_0^N)+\al K(\xi)u_0^N\Bigr]\dd \xi.
\end{eqnarray*}
Fix $\eps_5 < 1/4$. Then there exists a positive constant $A_3 $, independent of $s,\tau$,  $N$,
and $\al\in[0,\al_1]$, 
such that
\beq\label{ocut}
\|u_{tt}^N(s,\cdot)\|_{L^2(\Omega)} \le A_3.
\ee
Combining  (\ref{G3ttt2}) with \reff{eps34} and (\ref{ocut}) yields the bound
\beq\label{G3ttttt2}
\begin{array} {ll}
\displaystyle\| u_{tt}^N(\tau,\cdot)\|^2_{L^2(\Omega)}
+ a^2 \|\nabla u^N_t(\tau,\cdot)\|^2_{L^2(\Omega)} + \underline b \int_s^\tau\| \nabla u_{tt}^N(t,\cdot)\|^2_{L^2(\Omega)} \dd t \\\displaystyle \qquad\le
 \left( A_3+a^2 \|\nabla u^N_1 \|^2_{L^2(\Omega)}\!
\right.
\!\left.+\! \int_s^{\tau}  \| \partial_t F_0(t,\cdot)\|^2_{L^2(\Omega)}
 \dd t\right)\\ \displaystyle\qquad\quad\ \, + A_4\int_s^{\tau} \left(  \|u^N_{tt}(t,\cdot)\|^2_{L^2(\Omega)} + \|\nabla u^N_{t}(t,\cdot)\|^2_{L^2(\Omega)}
+\int_s^{t}\|\nabla u^N_{tt}(\theta,\cdot)\|^2_{L^2(\Omega)} \dd \theta \right) \dd t,
\end{array}
\ee
where the constant $A_4$ is fixed to satisfy the inequality
$$A_4 \ge \max\left\{1+L+\frac{1}{\eps_3},2\eps_2(\al)^2\eps_3+L_0\gamma^2+\frac{b_1^2}{\eps_4}\right\}
$$
for all $\al\in[0,\al_1]$.
  Now,  Gr\"onwall's argument applied to \reff{G3ttttt2} yields that 
there
exists a positive constant $A_5$, independent of $s\in\R$,$t \in [s,T]$, $N$, and
$\al\in[0,\al_1]$, such that
\begin{eqnarray}\label{utt}
\|u_{tt}^N(t,\cdot)\|_{L^2(\Omega)} +
\|\nabla u_{t}^N(t,\cdot)\|_{L^2(\Omega)}+\int_s^t\|\nabla u_{tt}^N(\tau,\cdot)\|_{L^2(\Omega)} \dd \tau
\le A_5.
\end{eqnarray}

{\it 2. Limit transition step and weak solutions.}
Since the estimates (\ref{prod})  and (\ref{utt}) are uniform in $N$, there exists a subsequence of $(u^N)$
(for simplicity, denoted  again by $(u^N)$) such that
the following convergences hold as $N\to\infty$:
\begin{equation} \label{G6z}
	\begin{array}{ll}
		u^N \overset{\ast}{\rightharpoonup} u   \ {\rm \ in} \ L^\infty(s,T; H^1_0(\Omega)), &\quad 
		u^N \overset{}{\rightharpoonup} u \ {\rm  \ in} \ L^2(s,T; H^1_0(\Omega)), \\[2mm]
		u^N_t \overset{\ast}{\rightharpoonup} u_t \ {\rm  \ in} \  L^\infty(s,T; H^1_0(\Omega)),&
		\quad
		u^N_t \overset{}{\rightharpoonup} u_t \ {\rm  \ in} \  L^2(s,T; H^1_0(\Omega)),\\[2mm]
		u^N_{tt} \overset{\ast}{\rightharpoonup} u_{tt} \ {\rm  \ in} \ L^\infty(s,T; L^2(\Omega)), &\quad
		u^N_{tt} \overset{}{\rightharpoonup} u_{tt} \ {\rm  \ in} \  L^2(s,T; H^1_0(\Omega)).
	\end{array}
\end{equation}
Then, due to  \cite[Lemma 1.2, p.7]{Lions},  the functions  $u$ and $u_{t}$
belong to $C([s,T]; H^1_0(\Omega)).$

Since the sequence $(u^N)$ is bounded in $L^2(s,T; H^1_0(\Omega))$ and  the sequence $(u^N_t)$ is 
bounded in $L^2(s,T; L^2(\Omega)),$  the sequence $(u^N)$ is bounded in
 $H^1(Q_{(s,T)})$. 
Moreover, the embedding $H^1(Q_{(s,T)})\hookrightarrow L^2(Q_{(s,T)})$ is compact by the Rellich-Kondrachov theorem.
Hence, there exists a subsequence of $(u^N)$ (still denoted by $(u^N)$) that converges strongly to $u$ in
$L^2(Q_{(s,T)})$ and, after passing to a further subsequence, almost everywhere in $Q_{(s,T)}$.  
Consequently, 
$$\varphi(x,u^N)\overset{}{\rightharpoonup} \varphi(x,u) \ \mbox{  in } L^2(Q_{(s,T)})
$$  as $N\to\infty$
by  \cite[Lemma 1.3]{Lions}.
Moreover, since $u$ is continuous  in $t$, the function $\varphi(x,u(x,t))$ is also continuous  in $t$.

Finally,  (\ref{G6z}) implies the weak convergences 
$$
\nabla u_t^N\overset{}{\rightharpoonup} \nabla u_t\ \mbox{  and }\ 
\nabla u^N\overset{}{\rightharpoonup} \nabla u \ \mbox{  in } L^2(s,T; L^2(\Omega))
$$ 
 as $N\to\infty$.
It follows that
$$\left[\Phi_\al u^N\right]\overset{}{\rightharpoonup} \left[\Phi_\al u\right]\ \mbox{  in } L^2(s,T; L^2(\Omega)).
$$

As a consequence of the derived convergences and the identities \reff{G2}, the  limit function~ $u$ satisfies (\ref{G1}).  
We thus conclude that $u$ is 
a weak solution to the problem (\ref{G0}), \reff{2},  \reff{2p} in $Q_{(s,T)}$ in the sense of Definition~\ref{def-obm2}.

{\it 3. Uniqueness of the weak solution.} 
 We prove the uniqueness in a standard way. Suppose, on the contrary,  that there exist two weak solutions 
 $u$ and $\tilde u$ to the problem
(\ref{G0}), \reff{2},  \reff{2p}. Then the function $z=u-\tilde u$ 
  satisfies the trivial initial conditions, namely $z(s,x)=0$ and $z_t(s,x)= 0$, as well as  the following  equality for all $v\in L^2(s,\tau; H_0^1(\Omega))$ and all $\tau\in (s,T]$:
    \beq \label{riv48g}
    \begin{array}{rr}
\displaystyle         \int_{Q_{(s,\tau)}}\Bigl(z_{tt} v+
                a^2\nabla z\cdot \nabla v +
            b(t)\nabla z_t \cdot \nabla v+
            (\varphi(x,u)-\varphi(x,\tilde u))v \\
            - \left(\left[\Phi_\al u\right](t,x) - \left[\Phi_\al \tilde u\right](t,x)\right)v
            \Bigl)\dd x    \dd t= 0.
    \end{array}
\ee
For $v=z_t$, the equation (\ref{riv48g}) reads as follows:
  \beq   \label{ed1obm}
 \begin{array}{rr}\displaystyle
        \int_{Q_{(s,\tau)}}
        \Bigl(z_{tt} z_t+
                a^2\nabla z\cdot \nabla z_t +
            b(t)\nabla  z_t \cdot \nabla z_t+
                        (\varphi(x,u)-\varphi(x,\tilde u))z_t\\
             -\left(\left[\Phi_\al u\right](t,x) - \left[\Phi_\al \tilde u\right](t,x)\right)z_t
            \Bigl)\dd x    \dd t= 0.
 \end{array}
\ee
Similarly to obtaining the estimate  (\ref{G44})  from the equality  (\ref{G3}), we derive the following estimate from \reff{ed1obm}:
\begin{equation} \label{G4obm}
\begin{array}{ll}\displaystyle\|z_{t}(\tau,\cdot)\|^2_{L^2(\Omega)} +
 \|\nabla z(\tau,\cdot)\|^2_{L^2(\Omega)}	+\int_{s}^\tau  \|\nabla z_t(t,\cdot)\|^2_{L^2(\Omega)} \dd t 	\\ [4mm]
 \displaystyle\qquad \le
  A_2\int_s^{\tau}\left( \|z_{t}(t,\cdot)\|^2_{L^2(\Omega)} +
 \|\nabla z(t,\cdot)\|^2_{L^2(\Omega)}+\int_s^{t}\|\nabla z_t(\theta,\cdot)\|^2_{L^2(\Omega)} \dd \theta \right) \dd t,
\end{array}\end{equation}
 where the constant $A_2$ is the same as in  (\ref{G44}).
Applying Gr\"onwall's argument to \reff{G4obm} yelds the inequality 
$$
\|z(\tau,\cdot)\|^2_{L^2(\Omega)} +\|z_{t}(\tau,\cdot)\|^2_{L^2(\Omega)} +
\|\nabla z(\tau,\cdot)\|^2_{L^2(\Omega)}	+\int_{s}^\tau  \|\nabla z_t(t,\cdot)\|^2_{L^2(\Omega)} \dd t  \le 0,
$$
being true for all $\tau\in[s,T]$. This implies the desired  uniqueness.

{\it 4. Strong solutions.} 
Let $\al\in[0,\al_1]$ be arbitrary fixed and let   $u$ be the corresponding weak solution to the problem  (\ref{G0}), \reff{2},  \reff{2p} in $Q_{(s,T)}$, obtained in Steps 1 and 2. 
For  almost every $t \in (s,T),$ consider the equality
\begin{equation} \label{utt2}
	\Delta (a^2 u + b(t) u_t) =  u_{tt} +  \varphi(x, u)  - \left[\Phi_\al u\right](t,x)- F_\al(t,x), \quad  x\in\Omega,
\end{equation}
which is satisfied in the weak sense according to (\ref{G1}).
In the notations 
\begin{equation} \label{ZZ}
	Z(x,t) =  u_{tt} +  \varphi(x, u)  -\left[\Phi_\al u\right](t,x)- F_\al(t,x)
\end{equation}
and 
\begin{equation} \label{zz}
	z(x,t) =  a^2 u + b(t) u_t,
\end{equation}
the equation \reff{utt2} takes the form
\beq\label{Zz} \Delta z = Z(t,x),
\ee
where $\Delta $ denotes the spatial Laplacian.
As $u$ is fixed, for simplicity we dropped the dependence on $u$ in the above notations.
By Step 2, $Z(t,\cdot)\in L^2(\Omega)$ and $z(t,\cdot)\in H_0^1(\Omega)$ for almost every $t\in(s,T)$. In addition, there exists a positive constant $A_6$,  independent of  $\al\in[0,\al_1]$,
such that
$$
\|Z(t,\cdot)\|_{L^2(\Omega)} \le A_6\quad\mbox{ for almost every }t \in (s,T).
$$
By considering the equation \reff{Zz} with respect to $z(t,\cdot)$ for each  $t \in [s,T]$
(with $t$ treated as a parameter), we conclude that $z(t,\cdot)$ is a weak solution to  \reff{Zz}. Since $\d\Omega\in C^2$, it follows from  \cite[Theorem 4, p. 317]{Evans} that
$z(t,\cdot)\in H^2(\Omega)$, for almost every $t \in (s,T)$. Moreover, the following estimate is satisfied for almost every $t \in (s,T)$:
\beq\label{ZZZ}\begin{array}{ll} 
	 \|z(t,\cdot)\|_{H^2(\Omega)} \le A_7\left(\| z(t,\cdot)\|_{L^2(\Omega)} +
	\| Z(t,\cdot) \|_{L^2(\Omega)}\right)   \\ [2mm]
	\qquad \le A_7\left(a^2 \| u(t,\cdot)\|_{L^2(\Omega)} + \bar b\| u_t(t,\cdot)\|_{L^2(\Omega)}
	+ \| Z(t,\cdot) \|_{L^2(\Omega)}\right) 
\end{array}
\ee
for a constant $A_7$ that is independent of $t$ and $\al\in[0,\al_1]$.

Next, we consider the  equation (\ref{utt2}) as a linear nonhomogeneous equation with respect to $u(\cdot,x)$, now with $x$ treated  as a parameter. The continuous solution $u(\cdot,x)$ to \reff{zz}
 can explicitly be given  by the formula
\begin{equation} \label{utt44}
u(t,x) = \exp\left\{-a^2\int_s^t \frac{\dd \theta}{b(\theta)}\right\} u(s,x) + \int_{s}^t  \exp\left\{-a^2\int_\theta^t \frac{\dd\theta_1}{b(\theta_1)}\right\}\frac{z(\theta,x)}{b(\theta)} \dd\theta,
\end{equation}
where $u(s,x)=u_0(x)$. 
As $u_0\in  H^2(\Omega)$ and $z(t,\cdot)\in H^2(\Omega)$ for all $t\in[s,T]$, we conclude  that 
$u\in C([s,T];H^2(\Omega))$
and that
$$\|  u(t,\cdot) \|_{H^2(\Omega)} \le A_8\quad \mbox{ for all } t \in [s,T],
$$
being true for a  positive constant $A_8$ chosen to be uniform in $\al\in[0,\al_1]$.  Moreover, because of  (\ref{ZZZ}),
$z\in L^\infty(s,T; H^2(\Omega))$. Then it follows from (\ref{zz}) that $u_t \in L^\infty(s,T; H^2(\Omega)).$
Consequently, $u$ satisfies the equation (\ref{G0}) almost everywhere. 
\qed\end{proof}

\subsubsection{Auxiliary problems on half-lines}\label{proof_aper}
Here we extend to the half-lines $[s,\infty)$ the strong solutions constructed in Subsection~\ref{interval} on finite time intervals.

\begin{defn}
	\label{def-semiaxis}
	Let $s\in\R$  be arbitrary real number.  A  function $u$  is  called a  {\rm bounded  strong solution}  
	to the initial-boundary value problem   (\ref{G0}), \reff{2},  \reff{2p} in $Q_{(s,\infty)}$ if 
		\beq\label{u-reg}
		\begin{array}{cc}
		u \in BC([s,\infty); H_0^1(\Omega)\cap H^2(\Omega)),\quad 
		u_t \in BC([s,\infty); H_0^1(\Omega)) \cap L^\infty (s,\infty; H^2(\Omega)), \\ [2mm]
		u_{tt}\in L^\infty (s,\infty; L^2(\Omega)),
\end{array}	
\ee
and	$u$ satisfies the initial conditions  (\ref{2p}) and the equation (\ref{G0}) almost everywhere.
\end{defn}

\begin{thm}  \label{lem41}
	Let $s\in\R$ be  fixed and  let 	$u_0, u_1 \in H_0^1(\Omega)\cap H^2(\Omega)$. 
	Suppose that  Conditions {(A1)}--{(A5)} are satisfied.
	If $L_0$,  $b_1$, and $\|\nabla K\|_{L^2(\Omega)}$
	are sufficiently small, then there exists $\al_2\in[0,\al_1]$ such that, for all $\al\in[0,\al_2]$, the problem   (\ref{G0}), \reff{2},  \reff{2p} in $Q_{(s,\infty)}$ 	 has a  unique bounded strong solution.
\end{thm}
\begin{proof} For each fixed $\al\in[0,\al_1]$, we construct a function $u$  on $Q_{(s,\infty)}$  such  that, for any $T>s$, its restriction  to $Q_{(s,T)}$ is the unique strong solution	to the problem (\ref{G0}), \reff{2},  \reff{2p} in $Q_{(s,T)}$ provided by Theorem \ref{T-obm}. 
	
Our goal is to show that the thus constructed function $u$ is a  bounded strong  solution on  $Q_{(s,\infty)}$ to  the problem  (\ref{G0}), \reff{2},  \reff{2p} in $Q_{(s,\infty)}$
	in the sense of Definition \ref{def-semiaxis}, provided that 
	the data $L_0$,  $b_1$,  $\|\nabla K\|_{L^2(\Omega)}$, and the parameter $\al$ are sufficiently small.
	
		The proof is divided into three claims.
	
	\begin{claim}\label{claim1}
	If $L_0$,  $b_1$, and $\|\nabla K\|_{L^2(\Omega)}$
	are sufficiently small, then there exists $\al_2\in[0,\al_1]$ such that, for all $\al\in[0,\al_2]$,
	the function $u$ satisfies the following regularity properties:
	$u \in BC([s,\infty), H_0^1(\Omega))$ and $u_t \in BC([s,\infty), L^2(\Omega)).$
\end{claim}
\begin{proof} Let $T>s$ be arbitrary fixed.	
For the $N$-th Galerkin's approximation $u^{N}$ of $u$ on $[s,T]$, we introduce the energy function
\beq\label{E}
{\mathcal {E}}^N(t) =  \frac{1}{2}\int_\Omega |u_{t}^N(t,x)|^2\dd x
+  \frac{a^2}{2}\int_\Omega  \|\nabla u^N(t,x)\|^2 \dd x.
\ee
For $t\in[s,T]$, rewrite \reff{G4pr} in terms of $\E^N$, as follows:
\beq \label{G5}\begin{array}{cc}
  \displaystyle 	 \mathcal E^N(t)
 + \int_{Q_{(s,t)}}  b(\tau)\|\nabla u_{\tau}^N\|^2 \dd x \dd\tau
 \displaystyle    + \int_{Q_{(s,t)}}\varphi(x,u^N)\, u_{\tau}^N \dd x \dd \tau
  \nonumber\\
 \displaystyle 	 =     \mathcal E^N(s)
 + \int_{Q_{(s,t)}}\Bigl(F_\al(\tau,x)+[\Phi_\al u^N](\tau,x)\Bigr)u_\tau^N\dd x \dd\tau.
\end{array}
\ee
 Hence,
 \beq\label{G50}
 \begin{array}{cc}
 \displaystyle   \frac{\dd \mathcal E^N(t)}{\dd t} = - \int_{\Omega} b(t)
 \|\nabla u_{t}^N\|^2 \dd x
 -  \int_{\Omega}\varphi(x,u^N) u_{t}^N \dd x  +  \int_{\Omega}\left(F_\al(t,x)+[\Phi_\al u^N](t,x)\right)u_t^N\dd x.
 \end{array}
\ee
The following equality is obtained from \reff{G2}, analogously to  (\ref{G3}):
$$
\begin{array}{cc} 
 \displaystyle  \int_{Q_{(s,\tau)}}\Bigl(u_{tt}^N u^N +  a^2\nabla u^N \cdot\nabla u^N +
b(t) \nabla u^N_{t} \cdot\nabla u^N  + \varphi(x,u^N) u^N  \Bigl)\dd x \dd t \\	
\displaystyle  = \int_{Q_{(s,\tau)}} \Bigl(F_\al(\tau,x)+[\Phi_\al u^N](\tau,x)\Bigr) u^N\dd x \dd t.
\end{array}
$$
Hence,  for almost every $t\in (s,T)$, we have
$$
\begin{array}{cc} 
 \displaystyle  \int_{\Omega}\Bigl(u_{tt}^N u^N +  a^2\nabla u^N \cdot\nabla u^N +
b(t) \nabla u^N_{t} \cdot\nabla u^N  + \varphi(x,u^N) u^N  \Bigl)\dd x \\	
  \displaystyle = \int_{\Omega} \Bigl(F_\al(t,x)+[\Phi_\al u^N](t,x)\Bigr) u^N\dd x,
\end{array}
$$
which implies
\beq\label{G6}
\begin{array}{cc}
 \displaystyle \frac{\dd}{\dd t}
\int_\Omega \left( u^N u_{t}^N dx+\frac{ b(t)}{2} 
\|\nabla u^N\|^ 2\right) \dd x 
= \int_\Omega (u_t^N)^2 \dd x +\left(\frac{b^\prime(t) }2-  a^2\right) \int_\Omega  \|\nabla u^N\|^2  dx\\ [4mm]
\displaystyle  - \int_\Omega  \varphi (x,u^N) u^N \dd x 
\displaystyle  + \int_\Omega \Bigl(F_\al(t,x)+[\Phi_\al u^N](t,x)\Bigr)u^N \dd x.
\end{array}
\ee
Consider the function 
\beq \label{W}
W^N_\eta(t) = \mathcal E^N(t) + \eta
\int_\Omega \left( u^N u_{t}^N +\frac{ b(t)}{2} 
\|\nabla u^N\|^ 2\right) \dd x
\ee
on $[s,T]$,
where $\eta>0$ will be chosen later to ensure exponential decay of 
 $W^N_\eta$, uniformly in~$N$, as $t\to\infty$. We use the bound
\begin{eqnarray} \label{WN-lower}
W^N_\eta(t) \ge\left( \frac{1-\eta}{2} \right)\left\|u_t^N(t,\cdot)\right\|^2_{L^2(\Omega)} +
\left(\frac{a^2+ (\underline b-\gamma^2)\eta}{2} \right) \left\|\nabla u^N(t,\cdot)\right\|^2_{L^2(\Omega)}\ge 0,
\end{eqnarray}
where the first inequality  follows directly from the definitions of $\mathcal  E^N(t)$ and $W^N_\eta(t)$, while the second one holds provided that $\eta$ lies in the range 
\beq\label{eta}
0<\eta \le \min\left\{1,a^2 / |\gamma^2 - \underline b|\right\}.
\ee
Below $\eta$ will be fixed in accordance with \reff{eta}.

In view of  (\ref{G50}) and (\ref{G6}), we obtain the following equality,
where $\eta_1<\eta$ is an auxiliary technical positive parameter:
$$
\begin{array}{ll}
  \displaystyle \frac{\dd W^N_\eta(t)}{\dd t}=\displaystyle -\eta_1    W^N_\eta(t)+\eta_1    W^N_\eta(t)
  - \int_{\Omega} b(t)\| \nabla u_{t}^N\|^2\dd x  - \int_{\Omega}\varphi(x,u^N) u_{t}^N \dd x \\ [5mm]
 \qquad \quad\displaystyle\displaystyle 
  + \int_{\Omega}\left(F_\al(t,x)+[\Phi_\al u^N](t,x)\right)u_t^N\dd x 
  +\eta\int_\Omega \Bigl((u_t^N)^2  +  \Bigl(\frac{b^\prime(t)}2-a^2 \Bigl) \|\nabla u^N(t,\cdot)\|^2\\ [5mm]
 \qquad \quad\displaystyle-  \varphi (x,u^N) u^N
 +  \left(F_\al(t,x)+[\Phi_\al u^N](t,x)\right)u^N  \Bigl)\dd x.
\end{array}
$$
Using  (\ref{phi}), we derive the estimates
\beq\label{dW00}
\begin{array}{rcl} 
\displaystyle  \int_{\Omega} \left|\left[\Phi_\al u^N\right](t,x)u^N\right| \dd x &\le& \displaystyle
\frac{1}{2}\eps_1(\al) \left((2\gamma^2 + 1)\|\nabla u^N(t,\cdot)\|_{L^2(\Omega)}^2+\|\nabla u_t^N(t,\cdot)\|_{L^2(\Omega)}^2
\right) \\[3mm]
\qquad &\le& \displaystyle \eps_6(\al) \left(\|\nabla u^N(t,\cdot)\|_{L^2(\Omega)}^2+\|\nabla u_t^N(t,\cdot)\|_{L^2(\Omega)}^2
\right) \\[5mm]
\displaystyle  \int_{\Omega} \left|\left[\Phi_\al u^N\right](t,x)u^N_t\right| \dd x &\le& \displaystyle
\frac{1}{2}\eps_1(\al) \left(\|\nabla u^N(t,\cdot)\|_{L^2(\Omega)}^2+(2\gamma^2 + 1)\|\nabla u_t^N(t,\cdot)\|_{L^2(\Omega)}^2
\right) \\[3mm]
\qquad &\le& \displaystyle \eps_6(\al) \left(\|\nabla u^N(t,\cdot)\|_{L^2(\Omega)}^2+\|\nabla u_t^N(t,\cdot)\|_{L^2(\Omega)}^2
\right),
\end{array} 
\ee
where 
$$
\eps_6(\al) = \displaystyle\frac{1}{2}\eps_1(\al)(2\gamma^2 + 1).
$$ 
Combining (\ref{riv16}),  \reff{W}, and (\ref{dW00}) with the inequalities:
$$\begin{array}{rcl}
\displaystyle \int_\Omega | F_\al(t,x)u^N|\dd x  &\le&\displaystyle \frac{\eps_7 \ga^2}{2}\int_\Omega (\nabla u^N)^2 \dd x + \frac{1}{2\eps_7} \sup_t
 \int_\Omega F_\al(t,x)^2 \dd x,  \\ [5mm]
 \displaystyle \int_\Omega | F_\al(t,x)u_t^N|\dd x  &\le&\displaystyle \frac{\eps_7 \ga^2}{2}\int_\Omega (\nabla u_t^N)^2 \dd x + \frac{1}{2\eps_7} \sup_t
 \int_\Omega F_\al(t,x)^2 \dd x,  \\ [5mm]
\displaystyle \int_\Omega | u^N u_t^N| \dd x &\le&\displaystyle \frac{\ga^2}{2}\int_\Omega (\nabla u^N)^2 \dd x + \frac{\ga^2}{2}\int_\Omega (\nabla u_t^N)^2 \dd x, \\ [5mm]
\end{array}
$$
valid for all $\eps_7>0$ and $\al>0$, 
we derive 
\beq\label{dW}
\begin{array}{ll} 
\hspace{-1mm}\displaystyle \frac{\dd W^N_\eta(t)}{\dd t} \displaystyle\le\displaystyle  -\eta_1 W^N_\eta(t)  + \frac{1+\eta}{2\eps_7} \sup_t  \left\|F_\al(t,\cdot)\right\|^2_{L^2(\Omega)} + \frac {\eta}{2}\biggl[ 
\eta_1\left(\frac{a^2}{\eta} +\ga^2+\overline{b}\right)
\\ [3mm]
\displaystyle \qquad-\left( 2 a^2-
b_1-\eps_6(\al)-2 L_0\ga^2- \ga^2\eps_7\right)+\frac{1}{\eta}\Bigl(\eps_6(\al)+L_0\ga^2\Bigr)
\biggl] \left\|\nabla u^N(t,\cdot)\right\|^2_{L^2(\Omega)} \\ [5mm]
\displaystyle\qquad + \frac{1}{2}\biggl[\eta_1\ga^2(1+\eta)- \left(2\underline b-  \gamma^2L_0-\eps_6(\al)- \ga^2\eps_7\right)+\eta
\left(\eps_6(\al)+2\ga^2
\right) 
 \biggr]\left\|\nabla u_t^N(t,\cdot)\right\|^2_{L^2(\Omega)}.
\end{array}
\ee
For sufficiently small $L_0$,  $b_1$,
and $\|\nabla K\|_{L^2(\Omega)}$, 
we can choose the parameters $\eta,$ $\eta_1$,  $\eps_7$, and~$\al$  in such a way that the  terms involving the factors $\|\nabla u^N\|_{L^2(\Omega)}$ and $\|\nabla u_t^N\|_{L^2(\Omega)}$ are non-positive. More precisely, we first assume that $L_0$,  $b_1$,
and $\|\nabla K\|_{L^2(\Omega)}$ are sufficiently small so that
$$ 2 a^2-
b_1-\eps_6(0)-2L_0\ga^2>0 \quad\mbox{and}\quad  
2\underline b-  \gamma^2L_0-\eps_6(0)>0.
$$
Then choose small enough $\al_2\in[0,\al_1]$ such that for all $\al\in[0,\al_2]$ the following inequalities are true:
$$ 2 a^2-
b_1-\eps_6(\al)-2L_0\ga^2>0 \quad\mbox{and}\quad  
2\underline b-  \gamma^2L_0-\eps_6(\al)>0.
$$
Afterwords fix first $\eps_7>0$ and then $B>0$ such that for all $\al\in[0,\al_2]$ it holds
$$
B\ge\min\Bigl\{ 2 a^2-b_1-\eps_6(\al)-2L_0\ga^2,2\underline b- 
 \gamma^2L_0-\eps_6(\al)\Bigr\}-\ga^2\eps_7>0.
$$ 
Next,  choose $\eta>0$ sufficiently small so that it satisfies not only \reff{eta} but also 
$$B-\eta\left(\eps_6(\al)+2\ga^2\right)>0$$ for all $\al\in[0,\al_2]$. Then fix $B_1>0$ such that the inequality 
\beq\label{B1}
B_1\ge B-\eta
\left(\eps_6(\al)+2\ga^2\right)
\ee
is true for all $\al\in[0,\al_2]$.
Finally,  choose  $\eta_1>0$ sufficiently small and refine the smallness assumptions on
$L_0$,  $b_1$,
 $\|\nabla K\|_{L_0^2(\Omega)}$, and $\al_2$ so that the bounds
$$
B_1-\eta_1\left(\frac{a^2}{\eta}+\ga^2+\overline{b}\right)-\frac{1}{\eta}\biggl(\eps_6(\al)+L_0\ga^2\biggr)>0\quad\mbox{and}\quad 
B_1-\eta_1\ga^2(1+\eta)>0
$$
hold for all  $\al\in[0,\al_2]$.
 
 Consequently,   (\ref{dW}) implies the inequality
$$ \frac{\dd W^N_\eta(t)}{\dd t} \le -\eta_1 W^N_\eta(t) + \frac{\eta + 1}{2\eps_7} \sup_t  \|F_\al(t,\cdot)\|^2_{L^2(\Omega)}.$$
Applying Gr\"onwall’s inequality then yields
\beq
	\label{WN-upper}
\displaystyle W^N_\eta(t) \le e^{-\eta_1(t-s)} W^N_\eta(s) + \frac{\eta + 1}{2\eps_7\eta_1}\sup_t  \|F_\al(t,\cdot)\|^2_{L^2(\Omega)},
\ee
which holds for all $t\in[s,T]$.
By (\ref{zb}), there exists a constant, depending only on the initial data $u_0$ and $u_1$, such that
$W^N_\eta(s)$ is bounded uniformly in $N$. 
Note that the estimate \reff{WN-upper} is uniform with respect to $s\in \R$,   $T\ge s$,  $\al\in[0,\al_2]$, and with respect to the Galerkin's approximations $u^N$ on $[s,T]$.
It follows that  $W^N_\eta(t)$ is uniformly bounded with respect to all the above parameters, and this bound remains valid for arbitrary
$T\ge s$. 
Combining this with \reff{WN-lower},  \reff{E}, \reff{W},  \reff{zb}, and \reff{G6z},  we obtain the following boundedness estimate for the function $u$ 
constructed at the beginning of the proof:
$$
\|u(t,\cdot)\|^2_{L^2(\Omega)} +\|u_t(t,\cdot)\|^2_{L^2(\Omega)} +
 \|\nabla u(t,\cdot)\|^2_{L^2(\Omega)}\le A_7,
$$
where the constant $A_7$  depends on the initial data $u_0$ and $u_1$,
but is independent of $s\in\R$, $t\ge s$, and  $\al\in[0,\al_2]$.  This completes the proof of the claim.
\qed\end{proof}

\begin{claim}\label{claim2}
If $L_0$,  $b_1$, and $\|\nabla K\|_{L^2(\Omega)}$
are sufficiently small, then the constant $\al_2$ can be chosen smaller, if necessary, such that, for every  $\al\in[0,\al_2]$, the constructed function $u$ possesses the  regularity \reff{u-reg}.
\end{claim}

\begin{proof}
Fix $T>s$ and introduce the function
\begin{eqnarray*}
 \mathcal F^N(t) =  \frac{1}{2}\int_\Omega |u_{tt}^N|^2\dd x
+  \frac{a^2 }{2}\int_\Omega  \|\nabla u_t^N\|^2\dd x,
\end{eqnarray*}
on $[s,T]$, where  $u^{N}$ is the $N$-th Galerkin's approximation of $u$. In this notation, the equality (\ref{G3ttt}) for $t\in[s,T]$ reads as follows:
$$
\begin{array}{cc}
	  \displaystyle 	  \mathcal F^N(t)
	+ \int_{Q_{(s,t)}}  \Bigl(b(\tau)\|\nabla u^N_{\tau\tau}\|^2+ b'(\tau)\nabla u^N_{\tau}\cdot\nabla u_{\tau\tau}^N
	 + \partial_2\varphi(x,u^N) u_{\tau}^Nu_{\tau\tau}^N \Bigr)\dd x\dd  \tau\\
\displaystyle 	=
	 \mathcal F^N(s) + \int_{Q_{(s,t)}}\partial_\tau \Bigl(F_\al(\tau,x)+[\Phi_\al u^N](\tau,x)\Bigr) u_{\tau\tau}^N \dd x \dd\tau.
\end{array}
$$
Hence,
\beq\label{G500}
\begin{array}{rcl} 
\displaystyle	 \frac{\dd \mathcal F^N(t)}{\dd t} &=&\displaystyle - \int_{\Omega} b(t)\|\nabla u_{tt}^N\|^2\dd x-\int_{\Omega}   b'(t)\nabla u^N_{t}\cdot\nabla u_{tt}^N\dd x
	\\ [5mm]
	 &&\displaystyle-  \int_{\Omega}\partial_2\varphi(x,u^N) u_{t}^N u_{tt}^N \dd x +
	  \int_{\Omega} \partial_t \Bigl(F_\al(t,x)+[\Phi_\al u^N](t,x)\Bigr)  u_{tt}^N  \dd x.
\end{array}
\ee

Along with the function $\mathcal F^N(t)$, consider its perturbed version,
defined by
\beq\label{V}
 V^N_\eta (t) =  \mathcal F^N(t) +
\eta \int_\Omega \left(u_t^N u_{tt}^N +
\frac{ b(t)}{2}\|\nabla u_t^N \|^2\right)\dd x,
 \ee
where $\eta$ is a positive parameter. Clearly,
$$
 V^N_\eta (t) \ge \left(\frac {1-\eta}{2}\right)\|u_{tt}^N(t,\cdot)\|_{L^2(\Omega)}+
\left(\frac{a^2+\underline b-\eta \gamma^2}{2}  \right) \|\nabla u_t^N(t,\cdot)\|_{L^2(\Omega)}\ge 0
$$
whenever 
$$
0<\eta \le \min\left\{1,(a^2 +\underline b) /  \gamma^2\right\}.
$$
In view of \reff{G500} and \reff{V}, we derive the following formula
for $\frac{\dd V_\eta^N(t)}{\dd t}$
(with an auxiliary technical parameter $\eta_1>0$):
$$
\begin{array}{ll}
\displaystyle \frac{\dd V_\eta^N(t)}{\dd t} = -\eta_1V_\eta^N(t)+\eta_1V_\eta^N(t)+\frac{\dd \mathcal F^N(t)}{\dd t}
 + \eta \frac{\dd }{\dd t}\int_\Omega \left(u_t^N u_{tt}^N +
\frac{ b(t)}{2}\|\nabla u_t^N \|^2\right)\dd x
  \\ [4mm]
\displaystyle\qquad =-\eta_1V_\eta^N(t) + \eta_1\biggl(\frac{1}{2}\int_\Omega\left( |u_{tt}^N|^2 +
a^2 \| \nabla u^N_t\|^2 \right) \dd x
+\frac{\eta}{2}\int_\Omega\left( 2u_t^N u_{tt}^N +
 b(t) \|\nabla u_t^N\|^2 \right) \dd x\biggr) \\ [4mm]
\displaystyle \qquad  \ \ \ - b(t)\int_{\Omega}\|\nabla u_{tt}^N\|^2\dd x-b'(t)\int_\Omega \nabla u_{t}^N \cdot \nabla u_{tt}^N \dd x
	-  \int_{\Omega}\partial_2\varphi(x,u^N) u_{t}^N u_{tt}^N \dd x\\ [4mm]
\displaystyle\qquad  \ \ \	+  \int_{\Omega} \partial_t \left(F_\al(t,x)+[\Phi_\al u^N](t,x)\right) u_{tt}^N \dd x  +  \eta\int_\Omega \left((u_{tt}^N)^2 
- \left(a^2 +  \frac{b^\prime(t)}{2}\right) \|\nabla u_t^N\|^2 \right)\dd x
	 \\ [4mm]
 \displaystyle \qquad  \ \ \ -  \eta\int_\Omega  \partial_2 \varphi (x,u^N) (u_t^N)^2 \dd x
	 +  \eta\int_\Omega \partial_t \left(F_\al(t,x)+[\Phi_\al u^N](t,x)\right)u_t^N \dd x.
\end{array}
$$
 Using the estimate (\ref{der}), we obtain the inequalities 
$$
\begin{array}{rcl} 
\displaystyle  \int_{\Omega} \left|\left[\partial_t\Phi_\al u^N\right](t,x)u^N_t\right| \dd x& \le& \displaystyle
\frac{1}{2}\eps_2(\al) \left((2\gamma^2 + 1)\|\nabla u^N_t(t,\cdot)\|_{L^2(\Omega)}^2+\|\nabla u_{tt}^N(t,\cdot)\|_{L^2(\Omega)}^2
\right) \\[2mm]
\qquad &\le& \displaystyle \eps_8(\al) \left(\|\nabla u^N_t(t,\cdot)\|_{L^2(\Omega)}^2+\|\nabla u_{tt}^N(t,\cdot)\|_{L^2(\Omega)}^2
\right), \\[4mm]
\displaystyle  \int_{\Omega} \left|\left[\partial_t\Phi_\al u^N\right](t,x)u^N_{tt}\right| \dd x &\le& \displaystyle
\frac{1}{2}\eps_2(\al) \left(\|\nabla u^N_t(t,\cdot)\|_{L^2(\Omega)}^2+(2\gamma^2 + 1)\|\nabla u_{tt}^N(t,\cdot)\|_{L^2(\Omega)}^2
\right) \\[2mm]
\qquad &\le& \displaystyle \eps_8(\al) \left(\|\nabla u^N_t(t,\cdot)\|_{L^2(\Omega)}^2+\|\nabla u_{tt}^N(t,\cdot)\|_{L^2(\Omega)}^2
\right).
\end{array} 
$$
where 
$$\eps_8(\al) = \displaystyle\frac{1}{2}\eps_2(\al)(2\gamma^2 + 1).
$$

Analogously to \reff{dW}, we now derive the estimate
\beq\label{dW1}
\begin{array}{ll}
 \displaystyle\frac{\dd V_\eta^N(t)}{\dd t} \le -\eta_1 V_\eta^N(t)
+\widetilde F^N_\al(t)\\ [5mm]
\displaystyle\qquad+\frac {\eta}{2}\left[ 
\eta_1\left(\frac{a^2}{\eta}+\overline{b}\right)- 
\left(2 a^2- b_1 - \frac{b_1}{\eta}\right)
+ 2 \eps_{8}(\al)\left(1+\frac{1}{\eta}\right)
\right] \left\|\nabla u^N_t(t,\cdot)\right\|^2_{L^2(\Omega)} \\ [5mm]
\displaystyle \qquad+ \frac{1}{2}\Bigl[\eta_1\ga^2(1+\eta)- 2\underline b
+ b_1 + \gamma^2(L_0 +\eps_9+ 2\eta) + 2(1 + \eta)\eps_{8}(\al)
\Bigr]\left\|\nabla u_{tt}^N(t,\cdot)\right\|^2_{L^2(\Omega)},
\end{array}
\ee
where $\eps_9>0$ is a positive constant and
$$\widetilde F_\al^N(t) =\displaystyle  \frac{1+ \eta}{2\eps_9}
\left\|\partial_t F_\al(t,\cdot)\right\|_{L^2(\Omega)}^2 
+ \frac{\eta}{2} \left(2\eta L_0 + L_0 + \eta  \eps_9 + \eta_1 \eta \right)\left\|u_t^N(t,\cdot)\right\|_{L^2(\Omega)}^2.$$
Similarly  to \reff{dW}, if $L_0 , b_1$, and $\|\nabla K\|_{L^2(\Omega)}$ are  sufficiently small, then the constants  $\eps_0$, $\al_2$, $\eta$, and $\eta_1$ can be chosen small enough so that the last two terms in \reff{dW1} (those involving $\|\nabla u_t^N\|_{L^2(\Omega)}$ and $\|\nabla u_{tt}^N\|_{L^2(\Omega)}$) are non-positive for all
$\al\in(0,\al_2]$. Consequently, the inequality (\ref{dW1}) implies the estimate
\begin{eqnarray*}
 \frac{\dd V_\eta^N(t)}{\dd t} \le -\eta_1 V_\eta^N(t) + \widetilde F_\al^N(t),
\end{eqnarray*}
which, by Gr\"onwall’s inequality, leads to
\begin{eqnarray} \label{*2}
\displaystyle V_\eta^N(t) \le e^{-\eta_1(t-s)} V_\eta^N(s) +\frac{1}{\eta_1}\widetilde F_\al^N(t),
\end{eqnarray}
being true for all $t\in[s,T]$.
Due to (\ref{zb}) and (\ref{ocut}), the function $V_\eta^N(s)$ is bounded uniformly with respect to $N$, independently of $s\in\R$, $T>s$, $\al\in[0,\al_2]$, and $u^N$. Combining this with Claim \ref{claim1}, we conclude that the  function $\widetilde F_\al^N(t)$ is uniformly bounded for all $t\in[s,\infty)$,  $N\in\N$, and $\al\in[0,\al_2]$.

The following estimate for $u$ now easily follows from \reff{G6z}, \reff{V}, and \reff{*2}:
\beq\label{uuuu}
\|u_{tt}(t,\cdot)\|^2_{L^2(\Omega)} +\|\nabla u_t(t,\cdot)\|^2_{L^2(\Omega)}\le A_8,
\ee
where the constant $A_8$  depends on the initial data $u_0$ and $u_1$,
but is independent of $s\in\R$, $t\ge s$, and $\al\in[0,\al_2]$.

To prove that $\Delta u \in BC([s,\infty); L^2(\Omega))$ and $\Delta u_t \in L^\infty(s,\infty; L^2(\Omega)),$
we follow the proof of Step 3 in the proof of Theorem \ref{T-obm} and conclude that there exists a positive constant $A_9$
such that the function
$Z(t,x)$, defined by (\ref{ZZ}),
fulfills the estimate
$\|Z(t,x)\|_{L^2(\Omega)} \le A_9$ for all $t\ge s$.
Therefore, because of  (\ref{zz}), (\ref{ZZZ}), and (\ref{utt44}), it holds
\beq\label{*3}
\|  u(t,\cdot) \|_{H^2(\Omega)} + \|  u_t(t,\cdot) \|_{H^2(\Omega)} \le A_{10}
\ee
 where  $A_{10}$ is a positive constant, which is independent of $s$, $t \ge s$, and $\al\in[0,\al_2]$.
\qed\end{proof}

\begin{claim}
The function $u$ constructed above is the unique bounded strong solution to the problem   (\ref{G0}), \reff{2},  \reff{2p} in $Q_{(s,\infty)}$.
\end{claim}
\begin{proof}
The existence follows from the above construction together with Claims~\ref{claim1} and \ref{claim2}. Uniqueness is a consequence of the uniqueness of strong solutions on finite time intervals provided by Theorem \ref{T-obm}.
\qed\end{proof}
The proof of Theorem \ref{lem41} is complete.
\qed\end{proof}

\subsubsection{Bounded strong solution on $\R$}\label{R}
Suppose that the data $L_0$,  $b_1$,  $\|\nabla K\|_{L^2(\Omega)}$, and the parameter $\al_2>0$ are sufficiently small so that the assumptions of Theorem~\ref{lem41} are fulfilled.

Here we construct a unique bounded strong solution on the whole line $\R$
 as the limit of solutions to auxiliary problems defined on the half-lines.
Specifically, for each $m \in \mathbb N$, we consider the following initial-boundary value problem in $Q_{(-m,\infty)}$
 corresponding to the equation
\begin{equation} \label{1m}
u_{tt} - a^2\Delta u -  b(t)\Delta u_t	
	+  \varphi(x,u)  = F^m_{\bar\al}(t,x)+\left[\Phi_{\bar\al}^m u\right](t,x),\quad t\in (-m,+\infty),\ x\in\Omega,
	\end{equation}
supplemented with the boundary conditions (\ref{2}) and the trivial initial data prescribed at $t=-m$:
\begin{equation}\label{2m}
u(-m,x) = 0, \quad u_t(-m,x) = 0,\qquad x\in\Omega.
\end{equation}
Here,  
$$
\begin{array}{cc}
	F_{\bar\al}^m(t,x) =  F_{\bar\al}(t,x) \ \ {\rm if} \ t > -m \quad
	{\rm and} \quad F_{\bar\al}^m(t,x) = 0 \ \ {\rm if} \ t \le -m,
	\\[2mm]
	\left[\Phi_{\bar\al}^mu\right](t,x)=\left[\Phi_{\bar\al}u\right](t,x)\ \ {\rm if} \ t > -m \quad
	{\rm and} \quad  \left[\Phi_{\bar\al}^mu\right](t,x)=0 \ \ {\rm if} \ t \le -m,
\end{array}
$$
and
$\bar{\al}$, being a fixed constant in $(0,\al_2]$, satisfies the smallness conditions of Theorem~\ref{lem41}.

According to Theorem \ref{lem41},  the problem (\ref{1m}), (\ref{2}), (\ref{2m}) in the domain $Q_{(-m,\infty)}$ 
admits a unique bounded  strong solution, denoted by $U^m(t,x)$. Note that this solution depends on~$\bar\al$.
 
 We extend $U^m$ to the whole axis $\R$ putting $U^m(t,x) = 0$ for $t < -m$.
Combining the bounds \reff{uuuu} and \reff{*3} yields the estimate
\beq\label{*4}
\|U^m\|_{L^\infty(\R;H^2(\Omega))}+\|U_t^m\|_{L^\infty(\R;H^2(\Omega))}+\|U_{tt}^m\|_{L^\infty(\R;L^2(\Omega))} \le C,
\ee
where the constant $C>0$ depends on $\al_2$  but is independent of $t\in\R$ and~$m\ge 1$.

The proof is divided into several claims. First, we establish existence  of a bounded weak solution
to the problem (\ref{G0}), \reff{2} with $\al=\bar\al$ in $Q$, as  the limit of the sequence~$(U^m)$ (denoted by $U$).
We then show that it enjoys additional regularity and thus, together with the function $g$ defined in  \reff{44}, qualifies as a bounded strong solution to the problem (\ref{1})--(\ref{3})  
in the sense of Definition~\ref{sol_weak} with $u=U$. The uniqueness result then completes the proof.

An important point to emphasize here is that, although the approximate solutions $U^m$ depend on $\bar{\al}$,  the limit function
$U$ does not.

We begin by introducing the notion of a bounded weak solution
to the problem (\ref{G0}), \reff{2}.

\begin{defn}
\label{def-wR}	 A function $U$  is  called a  {\rm  bounded weak solution} 
to the  problem  (\ref{G0}), \reff{2} with $\al=\bar{\al}$ if
\beq\label{U}
U \in BC(\mathbb R; H_0^1(\Omega)), \ \
U_t \in L^\infty(\mathbb R; H_0^1(\Omega)) \cap BC(\mathbb R; L^2(\Omega)), 
\ \ U_{tt}\in L^\infty(\R;L^2(\Omega)),
\ee
and the following equality holds for  all   $v\in C^1_c(\mathbb R, H_0^1(\Omega))$:
\begin{equation}
 \displaystyle\int\limits_{Q} \Bigl(U_{tt}v +  a^2\nabla U \cdot \nabla v + b(t) \nabla U_t \cdot \nabla v	
 +  \varphi(x,U)v -\left[\Phi_{\bar\al} U\right](t,x)v- F_{\bar \al}(t,x)v\Bigl)\dd x\dd t
   = 0.\label{G100}\end{equation}
   \end{defn}

\begin{claim}\label{convR}
	Let 
	$(U^m(t,x))$ be a sequence of functions  fulfilling the regularity conditions  \reff{u-reg} with $u=U^m$ and the estimates \reff{*4}. Then there exist a subsequence of $(U^m)$ (labeled by $m$ again) and a function
	$U$ fulfilling the regularity conditions \reff{U}
	and the following convergences:
	\beq\label{conv1}
		U^m\to U  \mbox{ in } C([-k,k];H^1(\Omega))\ \mbox{ and } \ U_t^m\to U_t   \mbox{ in }
		C([-k,k];L^2(\Omega))\ \mbox{ for every } k\in\N
		\ee
		and
		\beq\label{conv2}U_{tt}^m(t,\cdot)\rightharpoonup U_{tt}(t,\cdot) \ \mbox{ and } \
		\nabla U_{t}^m(t,\cdot) \rightharpoonup\nabla U_{t}(t,\cdot)    \mbox{ in } L^2(\Omega) \ \mbox{ for almost every } t\in\R.
		\ee		
\end{claim}

\begin{proof} 
The uniform bound \reff{*4} on	$(U_{t}^m)$ implies that the family $\{U^m\}$ is equicontinuous in time  with values in $H^1(\Omega)$. Similarly, the bound
\reff{*4} on	$U_{tt}^m$ ensures that $\{U_t^m\}$ is equicontinuous in time  with values in $L^2(\Omega)$. 
Moreover, since $(U^m(\cdot,t))$ is uniformly bounded in $H^2(\Omega)$ for all $t$ and the embedding $H^2(\Omega)\hookrightarrow H^1(\Omega)$ is compact, the set $\{U^m(\cdot,t)\}$ is relatively compact in $H^1(\Omega)$ for each fixed $t$. Therefore, by the Arzel\`a-Ascoli theorem in Banach spaces, for any fixed $k\in\N$ we can extract a subsequence (not relabeled) such that
\beq\label{conv1-0}
U^m\to U  \mbox{ in } C([-k,k];H^1(\Omega))\ \mbox{ and } \ U_t^m\to U_t   \mbox{ in }
C([-k,k];L^2(\Omega)).
\ee	
for some limit functions $U$ and $U_t$. Passing to the limit in the defining identity for $U_t^m$ shows that $U_t$ coincides with the distributional time derivative of $U$.
A standard diagonal argument with respect to $k$ extends these convergences to all $k \in \N$.

Furthermore, in view of \reff{conv1-0} and since the sequences $(U_{tt}^m)$ and $(\nabla U_{t}^m)$ are  bounded in 
$L^\infty(\R;L^2(\Omega))$, the Banach-Alaoglu theorem yields a subsequence of $(U^m)$ (still denoted by $(U^m)$) such that 
$$
U_{tt}^m\overset{\ast}{\rightharpoonup} U_{tt}\quad 
\mbox{and}\quad \nabla U_{t}^m\overset{\ast}{\rightharpoonup}\nabla U_{t}\ \ \mbox{ in }\ L^\infty(\R;L^2(\Omega)).
$$
In particular, for every $v\in L^1(\R;L^2(\Omega))$, it holds
\beq\label{c1}
\int_{\R}\left(U_{tt}^m(t,\cdot),v(t,\cdot)\right)_{L^2(\Omega)}\dd t\to
\int_{\R}\left(U_{tt}(t,\cdot),v(t,\cdot)\right)_{L^2(\Omega)}\dd t
\ee
and similarly for $\nabla U_{t}^m$, by the duality between  $L^1(\R;L^2(\Omega))$ and $L^\infty(\R;L^2(\Omega))$.

Taking $v(x,t)=\psi(x)\chi(t)$, with $\chi\in L^1(\R)$ and $\psi\in L^2(\Omega)$, in \reff{c1}, and using a standard density argument (based on the separability of $L^2(\Omega)$), we obtain the desired convergences \reff{conv2} for a (possibly further) subsequence of $(U^m)$.
\qed\end{proof}

\begin{claim}
The function $U$ constructed in the proof of Claim \ref{convR} is a bounded weak solution 
to the  problem  (\ref{G0}), \reff{2} with $\al=\bar{\al}$.
\end{claim}

\begin{proof}
According to Definition \ref{def-wR}, it remains to show that the weak formulation~\reff{G1} is true. Our starting point is that, by the definition of $U^m$, for a given $m\in\N$, the function  $U^m$
satisfies the following equalities
for any  $v\in C^1_c(\mathbb R; H^1_0(\Omega))$ (see \reff{G1}):
\beq\label{weakUm} \displaystyle\int_{Q} \Bigl(U^m_{tt}v + a^2\nabla U^m \cdot \nabla v +
b(t)\nabla U^m_{t} \cdot \nabla v	
+ \varphi(x,U^m)v -[\Phi_{\bar\alpha} U^m](t,x)v-F_{\bar\alpha}(t,x)v \Bigl)\dd x \dd t =0.
\ee
By Claim \ref{convR}, we can pass to the limit as $m\to \infty$ in this equality.	
Indeed, for the convergence of the second time derivative term, we use 
the first convergence in \reff{conv2} implying that $U_{tt}^m\rightharpoonup U_{tt}$  in $L^2(Q_{(-k,k)})$ for every $k\in\N$.  Next, the first convergence in \reff{conv1} implies that 
$\nabla U^m\to \nabla U$  in $L^2(Q_{(-k,k)})$ for every $k\in\N$. This justifies the convergence of the gradient term in \reff{weakUm}. For the damping term, we apply the second convergence in \reff{conv2}.

To treat the nonlinear term, note that by Assumption (A2) we have
$$
\left\|\vphi(\cdot,U^m(t,\cdot))-\vphi(\cdot,U(t,\cdot))\right\|_{L^2(\Omega)}\le L_0 \left\|U^m(t,\cdot)-U(t,\cdot)\right\|_{L^2(\Omega)}.
$$ 
Hence, the first convergence in \reff{conv1} implies the desired one
for every $k\in\N$:
$$
\vphi(x,U^m)\to \vphi(x,U)\quad\mbox{ in } C([-k,k];L^2(\Omega)).
$$

Finally, the convergence 
\beq\label{conv-Phi}
\int_Q \left[\Phi_{\bar\al} U^m\right](t,x) v\dd x \dd t \to \int_Q \left[\Phi_{\bar\al} U\right](t,x) v\dd x \dd t
\ee
as $m\to\infty$, now easily follows from the definition~\reff{FPhi}
of~$\Phi_{\alpha}$.  In more detail, without loss of generality, take $v(t,x)=\chi(t)\psi(x)$ with arbitrary fixed
$\psi \in H^1_0(\Omega)$ and $\chi \in C_c(\mathbb R)$. Then, due to Assumption (A4), the function $\chi(t)K(x)$ 
belongs to $C^1_c(\mathbb R; H^1_0(\Omega))$. Moreover, according  to~\reff{FPhi}, the left-hand side of \reff{conv-Phi}
can be written in the form
\begin{eqnarray*}
	& & \int_Q \left[\Phi_{\bar\al} U^m\right](t,x) \psi(x)\chi(t) \dd x \dd t =
	\int_\Omega K_0(x) \psi(x) \dd x \nonumber\\
	& & \times\int_Q \left[a^2\nabla K(\xi) \cdot  \nabla U^m+ b(t)\nabla K(\xi) \cdot \nabla U^m_{t}
	+ K(\xi)\varphi(\xi,U)+\bar\al K(\xi)U^m\right]\chi(t)\dd \xi \dd t.
\end{eqnarray*}
The desired convergence \reff{conv-Phi} now  follows immediately from Claim \ref{convR} and proceeds along the same lines as the convergence of all other terms in \reff{weakUm},  discussed above, but now with $v(t,x)=\chi(t)K(x)$.

Summarizing,  passing to the limit as $m\to\infty$  in \reff{weakUm}, we derive the desired weak formulation \reff{G100}.
\qed\end{proof}

\begin{claim}
	The bounded weak solution $U$  
	to the  problem  (\ref{G0}), \reff{2} with $\al=\bar{\al}$ admits additional regularity, namely
		$$U \in BC(\mathbb R; H_0^1(\Omega)\cap H^2(\Omega)), \
	U_t \in BC(\mathbb R; H_0^1(\Omega)) \cap L^\infty(\mathbb R; H^2(\Omega)), \
U_{tt}\in L^\infty(\R;L^2(\Omega)),$$ 
 and hence satisfies the equation (\ref{G0}) almost everywhere.
\end{claim}

\begin{proof}
Similarly to Step 4 in the  proof of Theorem \ref{T-obm}, the bounded weak solution $U$ satisfies the equality 
\begin{equation} \label{Utt2}
	\Delta (a^2 U + b(t) U_t) =  U_{tt} +  \varphi(x, U)  - \left[\Phi_{\bar\al} U\right](t,x)- F_{\bar\al}(t,x), \quad  x\in\Omega,
\end{equation}
 in a weak sense, namely in the sense of \reff{G100}. 
 Using the notations
 \begin{equation} \label{UUzz}
 	z(t,x) =  a^2 U + b(t) U_t
 \end{equation}
and
	$$Z(x,t) =  U_{tt} +  \varphi(x, U)  - \left[\Phi_{\bar\al} U\right](t,x)- F_{\bar\al}(t,x),$$
	the equation \reff{Utt2} takes the form
	\beq\label{Zx}
	\Delta z= Z(t,x).
	\ee
	By considering the equation with respect to $z(t,\cdot)$ for almost every $t\in\R$ (treated as a parameter), it follows from the regularity properties  \reff{U} that
 	$Z\in L^\infty(\mathbb R, L^2(\Omega))$ and $z\in L^\infty(\mathbb R, H_0^1(\Omega))$.

Moreover, as  $\d\Omega\in C^2$, applying  \cite[Theorem 4, p. 317]{Evans} to \reff{Zx} gives that for almost every 
$t\in \R$ the
solution $z(t,\cdot)$ 
belongs to $H^2(\Omega)$. Furthermore, 
the  estimate~\reff{ZZZ} is satisfied. This implies that $z\in L^\infty(\R;H^2(\Omega))$.

Further, we consider  (\ref{UUzz}) as a linear nonhomogeneous equation with respect to $U(\cdot,x)$ on $\R$, with $x$ treated as a parameter.
Then the bounded  continuous solution to the equation (\ref{UUzz}) is explicitly  given  by the formula
$$ U(t,x) = 
\int_{-\infty}^t  \exp\left\{-a^2\int_\theta^t \frac{\dd\theta_1}{b(\theta_1)}\right\}\frac{z(\theta,x)}{b(\theta)} \dd\theta.$$
This formula holds since, by the assumption  (A1), we have $ b(\theta_1)\ge  \bar b > 0$ for all $\theta_1\in\R$.
Since  $z(t,\cdot)\in H^2(\Omega)$, we conclude  that 
$U\in BC(\mathbb R, H^2(\Omega)).$
Moreover, on the account of $z\in L^\infty(\mathbb R, H^2(\Omega))$ and of (\ref{UUzz}), we get
 that $U_t \in L^\infty(\mathbb R, H^2(\Omega)).$
As a result, we obtain the estimate
$$
\|  U(t,\cdot) \|_{H^2(\Omega)} + \| U_t(t,\cdot) \|_{H^2(\Omega)} \le A_{11}
$$
for some $A_{11}>0$ and  all  $t \in \mathbb R$.
Hence, $U$ satisfies  (\ref{G0}) almost everywhere. 
\qed\end{proof}

\begin{claim}
	The pair of functions $(U,g)$, where $g$ is defined by \reff{44},  is a bounded strong solution 
	to the  problem  (\ref{1})--\reff{3}.
\end{claim}

\begin{proof} On account of Lemma \ref{lemma2.2}, it remains to show that the function $g(t)$, defined by the 
	formula \reff{44}, belongs to $BC(\R)$.  To this end, we  use 
 \cite[Theorem 4, p. 288]{Evans} stating that, if $V \in L^2(s,T;H^2(\Omega))$ with 
$V_t \in L^2(s,T;L^2(\Omega))$, where $T$ is some positive integer,
then $V \in C([s,T];H^1(\Omega))$ after possibly being redefined on a set of measure zero.
Moreover, the  estimate 
$$\max_{t \in [s,T]}\| V(t)\|_{H^1(\Omega)} \le C\left( \| V\|_{L^2( s,T;H^2(\Omega))} +
\| V_t\|_{L^2( s,T;L^2(\Omega))}\right),$$
holds, where the constant $C$ depends only on $s$, $T$, and $\Omega$. 
Since  $U_t \in L^\infty(\mathbb R;H^2(\Omega))$ and $U_{tt} \in L^\infty(\mathbb R;H^2(\Omega)),$
it follows that $U_t \in C([k,k+2];H^1(\Omega))$ for every $k \in \Z$, with a uniform bound on  $\|U_t(\cdot,t)\|_{H^1(\Omega)}$ independent of  $k$. Since the countable family of overlapping intervals
$[k,k+2]$, $k\in\Z$, covers $\R$,  we conclude that  $U_t \in BC(\mathbb R;H^1(\Omega)).$

Now, since all terms in (\ref{44}) are continuous functions
	on $\R$, the function  $g$ belongs to $C(\R)$, as desired.
\qed\end{proof}

\begin{claim}
The bounded strong solution $(U,g)$  
to the  problem  (\ref{1})--\reff{3} is unique.
\end{claim}

\begin{proof}
Let  $(\widetilde U,\tilde g)$ be another  bounded strong solution to the problem (\ref{1})--(\ref{3})
in~$Q$. Then, accordingly to Lemma \ref{lemma2.2}, the function  
$\widetilde U$ satisfies the equation \reff{4}  almost everywhere in $Q$
for some $\al>0$, and due to Remark \ref{all-al}, it satisfies \reff{4} 
for all $\al>0$, in particular, for $\al=\bar\al$. Hence, the function
$V=U-\widetilde U$ fulfills the equation 
     \begin{eqnarray*}
    V_{tt} - a^2\Delta V - b(t) \Delta V_t
			+  \varphi(x,U)-\varphi(x,\widetilde U) 
		 =[\Phi_{\bar\al} U](t,x)- [\Phi_{\bar\al} \widetilde{U}](t,x)   
    \end{eqnarray*}
almost everywhere in $Q$.
Write
	\begin{eqnarray*} 
		 \widetilde W_\eta(t) =\widetilde{\mathcal E}(t) + \eta \left( \int_\Omega V V_t \dd x + \frac{b(t)}{2}   \int_\Omega   \|\nabla V\|^2  \dd x \right),
	\end{eqnarray*}
where
$$
	\widetilde{\mathcal E}(t) =  \frac{1}{2}\int_\Omega |V_{t}|^2\dd x
+  \frac{ a^2}{2}\int_\Omega\|\nabla V\|^2 \dd x
$$
	and $\eta$ is a positive parameter satisfying \reff{eta} and \reff{B1}. This ensures that $\widetilde W_\eta(t) \ge 0$ for all $t \in \mathbb R.$
	As $\bar\al$   fulfills  conditions  of Theorem \ref{lem41}, then
	following  the proof of  Theorem~\ref{lem41},   we conclude that 
	$\widetilde W_\eta(t)$ satisfies the estimate
	\begin{eqnarray*}
	\displaystyle	\widetilde W_\eta(t) \le e^{{-\eta_1}(t-s)}	\widetilde W_\eta(s) \quad \mbox{for all } t \ge s
	\end{eqnarray*}
	with the same $\eta_1$ as in \reff{WN-upper}.
	Since for  any  $t\in\R$ and  any $\varepsilon>0$ there exists $s<t$ such that $e^{-{\eta}(t-s)} \le \varepsilon$, it follows that
	 $\widetilde W_\eta(t)  \equiv 0$  for all $t \in \mathbb R.$
	Consequently,  
$$
\|V_{t}(t,\cdot)\|^2_{L^2(\Omega)} +
 \|\nabla V(t,\cdot)\|^2_{L^2(\Omega)}  = 0\quad \mbox{for all } t\in\R,
 $$
  and     hence $V= 0$, as desired.
\qed\end{proof}
The proof of Claim 1 of Theorem \ref{main} is now complete.

\subsection{Bohr almost periodic solutions} \label{ap}

Let us recall the definition of a Bohr almost periodic function as given in \cite{Cord}.
 Let $(Y,\|\cdot\|_{Y})$ and $(Z,\|\cdot\|_{Z})$ be  Banach
spaces. A continuous function $f: \mathbb R \times Y \to Z$ is said to be {\it Bohr almost periodic} in $t\in\R$
uniformly in $x\in Y$ if for every
$\varepsilon > 0$ and every compact set ${\cal K} \subset Y$ there exists a number $l=l(\varepsilon,{\cal K}) > 0$ such
that any interval of length~$l$ contains a real number  $h$ (called an $\eps$-almost period)  such
that
$$\| f(t + h,y) - f(t,y) \|_Z < \varepsilon$$
for all $t \in  \mathbb R$ and all $y \in {\cal K}$.

Let $U$ be the bounded strong solution provided by Theorem \ref{main} (1). Accordingly to Lemma \ref{lemma2.2}, $U$ satisfies \reff{4}
for some $\al>0$.
 To prove that  $U$
is Bohr almost periodic in $t$, let $h$ be a common $\varepsilon$-almost period of functions $E(t), F_\al(t,x)$ and $b(t).$
Then the function 
$$
w(t,x) = U(t + h,x) - U(t,x)
$$ 
satisfies the following equation in the sense of Definition \ref{sol_weak}:
$$	\begin{array}{cc} 
		 w_{tt} - a^2\Delta w - b(t) \Delta w_t
		- (b(t + h) - b(t)) \Delta  U_t(t+h,x)) 	
		 \\[3mm]	=  \varphi(x,U(t,x))- \varphi(x,U(t + h,x))	
		 + F_\al(t+h,x)-F_\al(t,x)+[\Phi_\al U](t+h,x)- [\Phi_\al U](t,x).
	\end{array}$$
Setting 
	\begin{eqnarray*} 
	 W_h(t) &=& \mathcal E_h(t) + \eta \left( \int_\Omega w(t,x) w_t(t,x) \dd x +
	\frac{b(t)}{2}   \int_\Omega   \|\nabla w(t,x)\|^2  \dd x \right), \\
			 \mathcal E_h(t)& =&  \frac{1}{2}\int_\Omega |w_{t}(t,x)|^2\dd x
		+  \frac{ a^2}{2}\int_\Omega\|\nabla w(t,x)\|^2 \dd x,
	\end{eqnarray*}
and following the proof of  Theorem~\ref{lem41}, we obtain the following analog of the  inequality~\reff{WN-upper}:
	\begin{eqnarray*}
	 \displaystyle W_h(t) \le e^{{-\eta_1}(t-s)} W_h(s) +A_{11}\sup_{t}\left(\|F_\al(t+h,\cdot) - F_\al(t,\cdot)\|^2_{L^2(\Omega)} + |b(t + h) - b(t)|\right),
	\end{eqnarray*}
	where positive constants $\eta$,  $\eta_1$,  and $A_{11}$ are independent of $h$ and $t.$
	For any  $t\in\R$ and   $\varepsilon>0$,  there exists $s\le t$ such that $e^{-{\eta}(t-s)} \le \varepsilon.$
	This implies that
	$$\|U_t(t+h,\cdot) - U_t(t,\cdot)\|_{L^2(\Omega)} + \|U(t + h,\cdot) - U(t,\cdot)\|_{H^1(\Omega)}
	\le A_{12}\varepsilon$$
	for  some constant $A_{12}>0$,  independent of $h$, $t$, and $\eps$.
Consequently, $U(t,x)$ is almost periodic in $t$ with values in
 $H_0^1(\Omega)$, while  $U_t(t,x)$ is 
almost periodic  in $t$ with values in~$L^2(\Omega)$. 
	
The Bohr almost periodicity of   $g(t)$ follows from the 
formula (\ref{44}) with $u=U$, written as follows:
\beq\label{44ap}
 \begin{array}{rcl}
	    g(t)& =&\displaystyle  \left(\int_{\Omega}K(x)f_1(x)\dd x\right)^{-1}
	\bigg[E''(t) + \int_\Omega \Bigl( a^2\nabla K(x) \cdot  \nabla U \\ [3mm]
                 && - b(t)\Delta K(x) \cdot U_{t}
	+ K(x)\varphi(x,U) - K(x)f_2(x,t)\Bigl)\dd x\bigg]\ \mbox{for all }  t\in \mathbb{R}.
\end{array}
\ee
Indeed, since $E\in BC^3(\mathbb R)$, its second-order derivative  $\frac{\dd^2}{dt^2}E$ is almost periodic accordingly to~\cite[Theorem 1.8, p. 13]{Cord}. Combining this with the Bohr almost periodicity of 
$U_{t}$ and $\nabla U$, 
        the almost periodicity of $g$ then follows directly from  (\ref{44ap}).
        
        \subsection{ Periodic solutions} \label{per}
 To prove that, if the data of the original problem are $\om$-periodic, then the solution $U$ is also $\om$-periodic, we follow the proof of the almost periodicity of $U$ given in Section \ref{ap}, replacing almost periodicity by periodicity with $h=\om$ at each step. Note that the $BC^2(\mathbb R)$-regularity of the function $E$
(in contrast to the  $BC^3(\mathbb R)$-regularity required in the almost periodic case) is sufficient.
  
The proof of Claim 2, and thus of Theorem~\ref{main}, is complete.

\begin{rem}\rm
	As shown in the proof of Theorem~\ref{main}, the conditions of the theorem ensure the exponential stability of the obtained periodic and almost periodic solutions.
\end{rem}
 
 \section{Consequences and further extensions of our results}\label{extensions}
 
 In the course of our analysis, we also obtained several additional results on 
 inverse initial-boundary value problems for the strongly damped wave equation in bounded and unbounded  subdomains of $Q$. To formulate these results, we first introduce the corresponding notions of strong solutions.
 
 \begin{defn} \label{sol_strong} Let $s, T\in\R$ be arbitrary real numbers  satisfying  $s< T$.  
 	A pair of functions $(u(x,t),g(t))$  is  called a   {\rm   strong solution} in $Q_{(s,T)}$ (respectively, a   {\rm   bounded strong solution} in $Q_{(s,\infty)}$)  to the problem (\ref{1}), \reff{2}, (\ref{3}), \reff{2p}  if it satisfies the regularity conditions \reff{ug-reg},
 	with $\R$ replaced by $(s,T)$ (respectively, by $[s,\infty)$), the function $g(t)$ is determined by \reff{44}, and the function
 	$u$ fulfills the equations (\ref{1}) and \reff{2p} almost everywhere, as well as the overdetermination condition
 	\reff{3} pointwise.
 \end{defn}
 
Next, we present an equivalent formulation of the problem in the respective domains.
 
 \begin{lemma}\label{lem: equiv1}
 	A pair of functions $(u(x,t), g(t))$  is  a strong solution  in $Q_{(s,T)}$ (respectively, a    bounded strong solution in $Q_{(s,\infty)}$) 
 	to the problem   (\ref{1}), \reff{2}, (\ref{3}), \reff{2p} if and only if 
 	it satisfies the regularity conditions \reff{ug-reg},
 	with $\R$ replaced by $[s,T]$ (respectively, by $[s,\infty)$), the function $g(t)$ is determined by \reff{44}, and 
 	$u$ fulfills the equations (\ref{1}) and \reff{2p} almost everywhere.
 \end{lemma}

\begin{proof}
	The proof proceeds along the same lines as that of Lemma \ref{lemma2.2} with $\al=0$, with a minor modification in the sufficiency part. In this case, the equation \reff{E-E} with $\al=0$ yields $E^*(t) - E(t)=c_1t+c_2$, for some constants
	$c_1,c_2\in\R$. Since $u$ and $u^*$
	satisfy the same initial  conditions \reff{2p} and the same overdetermination condition \reff{3} (together with its differentiated form), it follows that $E^*(t) \equiv  E(t)$, as desired.
\qed	\end{proof}

Finally, we formulate the main result of this section.
 
\begin{thm}
	\label{T-obm-neobm}
	Let $s, T\in\R$ be arbitrary real numbers  satisfying  $s< T$. Moreover,  let 	$u_0\in H_0^1(\Omega)\cap H^2(\Omega)$,
	$u_1 \in H_0^1(\Omega)\cap H^2(\Omega)$, and   Conditions {(A1)}--{(A5)} be satisfied.  Then the following statements hold.
	\begin{itemize}
		\item[1.]
		The problem    (\ref{1}), \reff{2}, (\ref{3}), \reff{2p}  in $Q_{(s,T)}$
		has a unique  strong solution $(u,g)$.
				\item[2.]
				If $L_0$,  $b_1$, and $\|\nabla K\|_{L^2(\Omega)}$
				are sufficiently small,  the problem   (\ref{1}), \reff{2}, (\ref{3}), \reff{2p} in $Q_{(s,\infty)}$ 	 has a  unique bounded strong solution $(u,g)$.
	\end{itemize}
\end{thm}
\begin{proof}
	The first part follows from Lemma \ref{lem: equiv1} and Theorem \ref{T-obm} with $\alpha = 0$, and
	 the second part from Lemma \ref{lem: equiv1} and Theorem \ref{lem41} with $\alpha = 0$.
\qed\end{proof}

Note that Theorem \ref{T-obm-neobm} implies the result on weak solvability of the inverse problem (\ref{1}), (\ref{2}), (\ref{3}), (\ref{2p}) in $Q_{(s,T)}$ obtained in \cite{Aitzhanov,Protsakh3}. 
 
 \section*{Acknowledgment}
 Viktor Tkachenko was supported by the project 2025.07/0049 ``Complex dynamical systems: reduced order modelling in nature and society''
 of the National Research Foundation of Ukraine and by  Simons Foundation grant  (SFI-PD-Ukraine-00014586, V.T.).

\end{document}